\documentclass{amsart}

\usepackage{amsthm,amssymb,amsmath,hyperref}
\usepackage{graphicx}
\usepackage{color}
\usepackage{tikz}
\usepackage[utf8]{inputenc}

\theoremstyle{plain}
\newtheorem{thm}{Theorem}[section]

\newtheorem{cor}[thm]{Corollary}
\newtheorem{lem}[thm]{Lemma}
\newtheorem{prop}[thm]{Proposition}

\theoremstyle{definition}
\newtheorem{defn}[thm]{Definition}

\theoremstyle{remark}
\newtheorem{rem}[thm]{Remark}

\theoremstyle{plain}

\numberwithin{equation}{section}

\newcommand{\del}{\delta}

\newcommand{\B}{{\mathbb B}}
\newcommand{\D}{{\mathbb D}}

\newcommand{\R}{{\mathbb R}}

\newcommand{\C}{{\mathbb C}}

\newcommand{\N}{{\mathbb N}}

\newcommand{\Z}{{\mathbb Z}}

\newcommand{\re}{{\rm Re}\,}
\newcommand{\dist}{\hbox{ \rm dist}}

\newcommand{\calB}{{\mathcal B}}
\newcommand{\calC}{{\mathcal C}}
\newcommand{\calD}{{\mathcal D}}
\newcommand{\calE}{{\mathcal E}}

\newcommand{\calK}{{\mathcal K}}

\newcommand{\calM}{{\mathcal M }}

\newcommand{\calT}{{\mathcal T}}

\newcommand{\calV}{{\mathcal V}}

\newcommand{\frakC}{{\mathfrak C}}

\def\udot#1{\ifmmode\oalign{$#1$\crcr\hidewidth.\hidewidth
    }\else\oalign{#1\crcr\hidewidth.\hidewidth}\fi}

\def\R{\mathbb{R}}
\def\Z{\mathbb{Z}}

\def\C{\mathbb{C}}

\def\one{\mbox{1\hspace{-4.25pt}\fontsize{12}{14.4}\selectfont\textrm{1}}}

\begin{document}
	
\title[]{Dyadic Carleson embedding and sparse domination of weighted composition operators on strictly  pseudoconvex domains}

\author{Bingyang Hu}
\address{%
  Department of Mathematics, Purdue University, 150 N. University St.,W. Lafayette, IN 47907, USA}
\email{hu776@purdue.edu}

\author{Zhenghui Huo}
\address{Zhenghui Huo, Department of Mathematics and Statistics, The University of Toledo,  Toledo, OH 43606-3390, USA}
\email{zhenghui.huo@utoledo.edu}
\begin{abstract}
In this paper, we study the behavior of the weighted composition operators acting on Bergman spaces defined on strictly pseudoconvex domains via the sparse domination technique from harmonic analysis. As a byproduct, we also prove a weighted type estimate for the weighted composition operators which is adapted to Sawyer-testing conditions. Our results extend the work by the first author, Li, Shi and Wick under a much more general setting. 
\end{abstract}
\date{\today}
\maketitle

\section{Introduction}
Let $\Omega$ be a smooth, bounded, strictly pseudoconvex domain in $\C^n$ and $H(\Omega)$ be the set of all holomorphic functions in $\Omega$ with the usual compact open topology. For $0<p<\infty$, let $L^p=L^p(\Omega)$ be the collection of measurable functions $f$ in $\Omega$, for which the  $L^p$-norm
\begin{equation} \label{20210316eq01}
\|f\|_p:=\left( \int_\Omega |f(z)|^p dV(z) \right)^{\frac{1}{p}}
\end{equation}
is finite, where $dV$ is the Lebesgue measure. The \emph{Bergman space} $A^p=A^p(\Omega)$ on $\Omega$ is defined to be the space $L^p \cap H(\Omega)$. It is well known that when $1 \le p<\infty$, $A^p$ is a Banach space with the norm \eqref{20210316eq01}; while for $p \in (0, 1)$, it is a Fr\'echlet space with the translation invariant metric:
$$
d(f, g):=\left\| f-g \right\|_p^p, \quad f, g \in A^p.
$$
It is well known that when $p=2$, $A^2$ becomes a Hilbert space with inner product
$$
\langle f, g \rangle:=\int_\Omega f(z)\overline{g(z)} dV(z), \quad f, g \in A^2.
$$
Let $K(z, \xi), z, \xi \in \Omega$ be the \emph{Bergman kernel (reproducing kernel)} associated to the space $A^2(\Omega)$: namely,
$$
f(z)=\int_\Omega K(z, \xi) f(\xi)dV(\xi), \quad \forall f \in A^2. 
$$
The theory of Bergman spaces is a classical topic and is well understood for domains such as the unit disc and unit ball. One might consult the excellent books \cite{HKZ00, Zhu04} for more systematic treatments for these spaces. 

Let $u \in H(\Omega)$ and $\varphi: \Omega \to \Omega$ be a holomorphic self-mapping. The \emph{weighted composition operator} is defined as
$$
W_{u, \varphi}(f)(z)=u(z) \cdot f \circ \varphi(z), \quad f \in H(\Omega), \ z \in \Omega. 
$$
When $u(z) \equiv 1$, then $W_{u, \varphi}$ becomes the \emph{composition operator} and is denoted by $C_\varphi$, and if $\varphi(z)=z$, then $W_{u, \varphi}$ becomes the \emph{multiplication operator} and is denoted by $M_u$. Note that here we do not require $\varphi$ to be a bi-holomorphic mapping on $\Omega$, and hence $W_{u, \varphi}$ in general does not coincides with $M_u$. We refer the interested reader \cite{CZ04, L95, S96} and the references therein for more information about these operators.

In recent years, the sparse domination technique was developed and considered by many mathematics working in harmonic analysis. This technique dates back to Lerner's alternative and simple proof of the $A_2$ theorem \cite{L13a, L13b}, which was proved originally by Hyt\"onen \cite{H12}. One key feature in Lerner's approach was to show that all Calder\'on-Zygmund operators can be bounded by a special collection of dyadic, positive operators called \emph{sparse operators}. This estimate led almost instantly to a proof of the sharp dependence of the constant in related weighted norm inequalities, that is, the $A_2$ theorem, which has been actively worked on for over a decade. After Lerner's work, there have been many improvements and extensions of his ideas to a wide range of spaces and operators, such as \cite{CR16, CDO18, L17, LN}. 

Sparse domination technique in complex function theory is a recent research topic. To our best knowledge, the first appearance of such an estimate appears in the work of Aleman, Pott and Reguera \cite{APR17}, in which, the Sarason conjecture on the Bergman spaces was studied via a pointwise sparse domination estimate of the Bergman projection. Later, by using similar idea, Rahm Tchoundja and Wick \cite{RTW17} were able to prove a sharp weighted estimate for the Berezin transform and Bergman projection (see, also \cite{HWW20b, HW20a, HW20, GHK20} for some extension of this work). In a recent paper
\cite{HLSW19}, the authors extended this line of research by studying the behavior of weighted composition operators on the upper half plane and the unit ball via the sparse domination technique. In particular, new characterizations of the boundedness and compactness of the weighted composition operators are considered; moreover, they also gave a new weighted estimate associated to these operators. 

The goal of this paper is to extend the results in \cite{HLSW19} to the weighted composition operators acting on smooth, bounded, strictly  pseudoconvex domains. Here, we would like to include a comparison between the current case and the case when $\Omega$ is the upper half plane $\R_+^2$ considered in \cite{HLSW19}.

\begin{enumerate}
    \item [(1).] The dyadic decomposition on the strictly  pseudoconvex domain is more delicate. When $\Omega=\R_+^2$, such a dyadic decomposition can be  constructed by extending (in the sense of Carleson tent on $\R_+^2$) any dyadic grids on $\R$; while for the case when $\Omega$ is a strictly  pseudoconvex domain, one need to use local holomorphic coordinate system and the orthogonal projection map (see, Lemma \ref{proj}) near $\mathbf b\Omega$ to construct such tents (see, Section 2 for more detail for this construction);
    
    \medskip
    
    \item [(2).] In \cite{HLSW19}, we have used the Hardy-Littlewood maximal operator (as well as the fractional Hardy-Littlewood maximal operator) to control the contribution of the average of $f$ on each Carleson tent, however, this is not the case when $\Omega$ is strictly  pseudoconvex. This is because the Kobayashi balls (see, \eqref{Kball}) are not doubling while the Euclidean balls in $\R^2$ are.  We overcome this difficulty by showing the collection of all ``tents" (that is the dyadic decomposition associated to $\Omega$) forms a Muckenhoupt basis, which then guarantees the strong $(p, p)$ estimate for the maximal operator associated to this basis (see, Theorem \ref{Mbasis}). We also prove the fractional maximal operator associated to such a basis also enjoys certain $L^p \to L^q$ estimate (see, Lemma \ref{fractional}). 
\end{enumerate}

We next make a remark about our choice of the domains. Recall that in \cite{HWW20a}, the pointwise for sparse bounds of the Bergman projection are actually valid when the domain $\Omega$ is a simple domain, which contains the smooth, bounded and strictly pseudoconvex domain as a particular case. We only consider strictly pseudoconvex domains here because we need the following off-diagonal estimate of the Bergman kernel function for $(z,w)$ near the boundary diagonal: %The reason for us to only consider the strictly pseudoconvex domain is that, to our best knowledge, it is the only domain (among simple domains) which possesses an asymptotic expansion in the sense of Fefferman (see, e.g., \cite{Fefferman}). This further gives
\begin{equation} \label{eq2020}
|K(z, w)| \simeq K(z, z), \quad z, w \in \Omega. 
\end{equation} 
This estimate follows easily from the asymptotic expansion of the Bergman kernel on a strictly pseudoconvex domain \cite{Fefferman,BS76}. Yet, it is not clear if \eqref{eq2020} holds true on other classes of domains, such as finite type domains in $\mathbb C^2$ and convex domains of finite-type. The estimate \eqref{eq2020} plays a fundamental role in the classical complex function theory. See, for example,  \cite[Lemma 2.20]{Zhu04} for the case when $\Omega$ is the unit ball in $\C^n$. 

It is worth noting that our approach can also be extended to study the behavior of weighted composition operators on the weighted Bergman spaces $A^p_\alpha (\Omega)$ over strictly pseudoconvex domains, which are the collections of all holomorphic functions on $\Omega$ with
$$
\|f\|_{p, \alpha}^p:= \int_\Omega |f(z)|^p \textrm{dist}^\alpha (z, \mathbf b\Omega)dV(z)<\infty, 
$$
for $p \ge 1$ and $\alpha>-1$. This is because an asymptotic expansion for the weighted Bergman kernel of the weighted Bergman space $A^2_\alpha(\Omega)$ is still valid (see, e.g., \cite{E08}). 

\medskip

We organize the paper is as follows. In Section 2, we provide backgrounds on the construction of a dyadic decomposition on a strictly pseudoconvex domain, moreover, we also study the relation between the dyadic structure and the Kobayashi balls near the boundary $\mathbf b\Omega$; Section 3 proves a Carleson embedding type theorem associated to the dyadic decomposition. In Section 4, we first give new necessary and sufficient conditions for the weighted composition operators to be bounded and compact on weighted Bergman spaces via the sparse domination technique. Moreover, we establish a new weighted estimate which adapts Sawyer's classical test condition. 

Throughout this paper, for $a, b \in \R$, $a \lesssim b$ ($a \gtrsim b$, respectively) means there exists a positive number $C$, which is independent of $a$ and $b$, such that $a \leq Cb$ ($ a \geq Cb$, respectively). Moreover, if both $a \lesssim b$ and $a \gtrsim b$ hold, then we say $a \simeq b$.  

\medskip

\noindent \textbf{Acknowledgements.} The authors thank Brett Wick for useful discussions and suggestions on improving this paper. 

\bigskip

\section{Dyadic structure on strictly  pseudoconvex domain} \label{Sec02}

In this section, we first perform a dyadic decomposition for the strictly  pseudoconvex domain $\Omega$ and then study the sub-mean value property for such a structure (which is referred as ``kube" in the sequel). The main ingredient to prove this property is to show that such a dyadic decomposition ``admits" well with the Kobayashi balls near the boundary $\mathbf b\Omega$.

%------------------------------------
\subsection{Dyadic decomposition on $\Omega$.}

For this part, we mainly follow the framework built in \cite{HWW20a, HWW20b, HW20}. Heuristically, the construction of a dyadic decomposition on $\Omega$ consists of two steps:
\begin{enumerate}
    \item [I.] $\mathbf b \Omega$ is a space of homogeneous type, and hence it admits a dyadic decomposition;
    \item [II.] Using the dyadic structure on $\mathbf b \Omega$, one can associate a tree structure to $\Omega$. 
\end{enumerate}

We now turn to some details. 

\begin{defn}
A \emph{space of homogeneous type} is an ordered triple $(X, \rho, \mu)$, where $X$ is a set, $\varrho$ is a quasi-metric, that is
\begin{enumerate}
    \item [(1).] $\varrho(x, y)=0$ if and only if $x=y$;
    \item [(2).] $\varrho(x, y)=\varrho(y, x)$ for all $x, y \in X$;
    \item [(3).] $\varrho(x, y) \le \kappa (\varrho(x, z)+\varrho(y, z))$ for all $x, y, z \in X$,
\end{enumerate}
for some constant $\kappa>0$, and the non-negative Borel measure $\mu$ is doubling, that is, 
$$
0<\mu(B(x, 2r)) \le D\mu(B(x, r))<\infty, \quad \textrm{for some} \ D>0,
$$
where $B(x, r):=\{y \in X, \varrho(x, y)<r\}$, for $x \in X$ and $r>0$.
\end{defn}
The first step is to show that $\mathbf b\Omega$ admits a structure of space of homogeneous type. Such a construction is essentially contained in a series work in \cite{NRSW3, McN94a, McN94b, McN03}, which we recall below.  Let $\Omega$ be a strictly  pseudoconvex domain with defining function $\rho$ such that $\left| \nabla \rho(\xi) \right|=1$ for all $\xi \in \mathbf b\Omega$. Let further, $U$ be a small neighborhood of $\mathbf b\Omega$. For each $q \in U$ and $r>0$ sufficiently small, there exists a holomorphic coordinate system $z=(z_1, \dots, z_n)$ centered at $q$, such that 
$$
D(q, r):= \left\{ z \in \C^n: |z_j|<\tau_j(q, r),  \ j=1, \dots, n \right\}
$$
is the largest polydisc centered at $q$ on which $\rho$ changes by no more than $r$ from its value at $q$, that is, if $z \in D(q, r)$, then $|\rho(z)-\rho(q)|<r$. Here, 
$$
\tau_1(q, r)=r \quad \textrm{and} \quad \tau_j(q, r) \approx r^{1/2}, \ j=2, 3, \dots, n. 
$$
The polydisc $D(q, r)$ satisfies certain ``covering properties" (see, \cite{McN94a}):
\begin{enumerate}
    \item [(1).] There exists a constant $C>0$, such that for any $q_1, q_2 \in U \cap \Omega$ with $D(q_1, r) \cap D(q_2, r) \neq \emptyset$ with $r>0$, we have
    $$
    D(q_2,r) \subset CD(q_1, r) \quad \textrm{and} \quad D(q_1, r) \subseteq CD(q_2, r);
    $$
    \item [(2).] There exists a constant $c>0$ such that for $q \in U \cap \Omega$ and $r>0$, we have
    $$
    D(q, 2\del) \subset cD(q, \del). 
    $$
\end{enumerate}
By a partition of unity argument in \cite{McN94a}, $D(p, \del)$ induces a global quasi-metric on $\Omega$.

Moreover, these polydiscs induce a quasi-metric on $b\Omega$. More precisely, for any $z \in \mathbf b\Omega$ and $r>0$ sufficiently small, we let
\begin{equation} \label{bball}
B(z, r):=D(z, r) \cap \mathbf b\Omega
\end{equation}
and this allows us to define
$$
d(z, z'):=\inf \left\{r>0, z' \in B(z, r)\right\}, \quad \textrm{for any} \quad z, z' \in \mathbf b\Omega. 
$$
Note that the quasi-metric balls $B(q, r)$ also enjoy the ``covering properties" as $D(q, r)$:
\begin{enumerate}
    \item [(1).] There exists a constant $C>0$, such that for any $z_1, z_2 \in U \cap \mathbf b\Omega$ with $B(z_1, r) \cap B(z_2, r) \neq \emptyset$ with $r>0$, we have
    $$
    B(z_2,r) \subset CB(z_1, r) \quad \textrm{and} \quad B(z_1, r) \subseteq CB(z_2, r);
    $$
    \item [(2).] There exists a constant $c>0$ such that for $z \in U \cap \mathbf b\Omega$ and $r>0$, we have
    $$
    B(z, 2\del) \subset cB(z, \del). 
    $$
\end{enumerate}

\begin{thm} \label{McNeal}
Let $\Omega$ be a strictly  pseudoconvex domain. Then the triple $(\mathbf b \Omega, d, \sigma)$ becomes a space of homogeneous type, where $d\sigma$ is the normalized surface measure on $\mathbf b \Omega$. 
\end{thm}

Note that for any $z \in b\Omega$ and $r>0$, we can also write $B(z, r)$ to be the collection $\left\{w \in \mathbf b \Omega: d(z, w)<r \right\}$. 

\medskip

As a consequence of Theorem \ref{McNeal}, we are able to decompose $\mathbf b \Omega$ in a dyadic way. More precisely, we have the following result.

\begin{thm} [\cite{HK12}] \label{HKdecomp}
There exists a family of sets $\calD=\bigcup\limits_{k \ge 1} \calD_k$, called a dyadic grid associated to $b \Omega$, constants $\frakC>0, 0<\del, \epsilon<1$, and a corresponding family of points $\{c(Q)\}_{Q \in \calD}$, such that 
\begin{enumerate}
    \item [(1).] $\mathbf b \Omega=\bigcup\limits_{Q \in \calD_k} Q$, for all $k \in \Z$; 
    \item [(2).] For any $Q_1, Q_2 \in \calD$, if $Q_1 \cap Q_2 \neq \emptyset$, then $Q_1 \subset Q_2$ or $Q_2 \subset Q_1$;
    \item [(3).] For every $Q \in \calD_k$, there exists at least one child cube $Q_c \in \calD_{k+1}$ such that $Q_c \subset Q$; 
    \item [(4).] For every $Q \in \calD_k$, there exists exactly one parent $\hat{Q} \in \calD_{k-1}$ such that $Q \subset \hat{Q}$;
    \item [(5).] If $Q_2$ is an offspring of $Q_1$, then $\sigma(Q_2) \ge \epsilon\sigma(Q_1)$;
    \item [(6).] For any $Q \in \calD_k$, $B(c(Q), \del^k) \subset Q \subset B(c(Q), \frakC \del^k)$. 
\end{enumerate}
\end{thm}

The last property in Theorem \ref{HKdecomp} is referred as the \emph{sandwich property}, and each $Q \in \calD$ is referred as a \emph{dyadic cubes} with center $c(Q)$ and sidelength $\ell(Q)=\del^k$. We make a remark that these cubes are in general not standard Euclidean cubes in $\R^n$ or any quasi-metric balls $B(z, r)$ induced by $d$. Moreover, we may also assume $\del$ is sufficiently small (otherwise we could replace $\del$ by $\del^{N_0}$ for some $N_0 \in \N$ is sufficiently large). To this end, for the purpose of simplicity, for each dyadic grid $\calD$ on $\mathbf b\Omega$, we let $\calD^0$, which is the $0$-th generation, to be the collection $\{\mathbf b\Omega\}$ and we refer it has ``sidelength" $1$. 

\begin{prop} [{\cite[Theorem 4.1]{HK12}}] \label{adjac}
There exists a finite collection of dyadic grids $\calD^t, t=1, \dots, K_0$, such that for any quasi-metric ball $B=B(\xi, r) \subset \mathbf b \Omega$, there exists a dyadic cube $Q \in \calD^t$, such that
$$
B \subseteq Q \quad \textrm{and} \quad \ell(Q) \le \widetilde{\frakC}r,
$$
where $\frakC$ is an absolute constant which only depends on $\Omega$. Moreover, the constants $\frakC_t, \del_t$ and $\varepsilon_t$ constructed in Theorem \ref{HKdecomp} can be taken to be the same, that is $\frakC_1=\frakC_2=\dots=\frakC_{K_0}$, $\del_1=\del_2=\dots=\del_{K_0}$ and $\epsilon_1=\epsilon_2=\dots=\epsilon_{K_0}$. 
\end{prop}

To this end, we refer the dyadic grids $\calD^t, t=1, \dots, K_0$ the \emph{adjacent grids} associated to $\Omega$.

\medskip

Our next step is to extend the balls or cubes on $\mathbf b \Omega$ inside $\Omega$ by inducing ``tents" associated to them, this procedure is a generalization of the Carleson tent in both upper half plan or unit disc (ball) setting (see, e.g. \cite{Zhu04}). The key ingredient in defining such structures is to make use of the orthogonal projection near $\mathbf b \Omega$, which we now recall. 

For $\varepsilon>0$ sufficiently small, set
$$
N_\varepsilon(\mathbf b \Omega):=\{\omega \in \C^n: \dist(\omega, \mathbf b \Omega)<\varepsilon \},
$$
where $\dist(\cdot, \cdot)$ is the Euclidean distance in $\C^n$. The following Balogh-Bonk projection is important. 

\begin{lem} [\cite{BB00}] \label{proj}
There exists a sufficiently small $\varepsilon_0>0$ and a map $\pi: N_{\varepsilon_0}(\mathbf b \Omega) \to \mathbf b \Omega$ such that
\begin{enumerate}
    \item [(1).] For each point $z \in N_{\varepsilon_0}(\mathbf b \Omega)$ there exists a unique point $\pi(z) \in \mathbf b \Omega$ such that
    $$
    \dist(z, \pi(z))=\dist(z, \mathbf b \Omega);
    $$
    \item [(2).] For $p \in \mathbf b \Omega$, the fiber $\pi^{-1}(p)=\left\{p-\varepsilon n(p): -\varepsilon_0 \le \varepsilon<\varepsilon_0 \right\}$, where $n(p)$ is the outer unit normal vector of $\mathbf b \Omega$ at point $p$;
    \item [(3).] $\pi$ is smooth on $N_{\varepsilon_0}(\mathbf b \Omega)$;
    \item [(4).] If the defining function $\rho$ is the signed distance to the boundary $\mathbf b \Omega$, then we have
    $$
    \nabla \rho(z)=n(\pi(z)) \quad \textrm{for} \ z\in N_{\varepsilon_0} (\mathbf b \Omega). 
    $$
\end{enumerate}
\end{lem}

\begin{defn} \label{extension}
Let $\varepsilon_0$ and $\pi$ be defined as in Lemma \ref{proj}. Let further, $S$ be either a quasi-metric ball $B$ or a dyadic cube $Q$ on $\mathbf b \Omega$ with radius (or sidelength, respectively) $\ell(S)$. We define the \emph{tent} associated to $S$ as follows: 
\begin{enumerate}
    \item [(1).] If $\ell(s)<\varepsilon_0$, then 
    $$
    T(S):=\left\{\omega \in \Omega: \pi(\omega) \in S, \ \dist(\pi(\omega), \omega) \le \ell(S)\right\};
    $$
    \item [(2).] If $\ell(s) \ge \varepsilon_0 $, then $T(S):=\Omega$. In particular, we let $T(\mathbf b \Omega)=\Omega$. 
\end{enumerate}
\end{defn}

As a consequence of Proposition \ref{adjac}, we have the following result.

\begin{cor} [{\cite[Lemma 2.6]{HW20}}] \label{adjaccor}
Let $B$ and $Q$ be defined as in Proposition \ref{adjac}, then we have
$$
\emph{\textrm{Vol}}(T(B)) \simeq \emph{\textrm{Vol}} (T(Q)),
$$
where the implicit constant above is independent of the choice of $B$ and $Q$. 
\end{cor}

We are now ready to introduce the tree structure inside $\Omega$. To begin with, we may assume that the constant $\del$, which is defined in Theorem \ref{HKdecomp}, is smaller than $\varepsilon_0$. Note that this is to guarantee the extension of each dyadic cubes on $\mathbf b \Omega$ is non-trivial (namely, not equal to $\Omega$). 

For any dyadic grid $\calD$ on $\mathbf b \Omega$, let $\calT$ to be its \emph{dyadic extension} inside $\Omega$, namely, 
$$
\calT:=\bigcup_{k \ge 1} \bigcup_{j=1}^{J_k} T(Q_j^k) \cup \Omega,
$$
where $J_k \ge 0$ is the number of dyadic cubes in the $k$-th generation in $\calD$ and $Q_j^k \in \calD_k$. It is clear that for $Q_1, Q_2 \in \calT$ with $\ell(Q_1)<\ell(Q_2)$, we have either $T(Q_1) \subset T(Q_2)$ or $T(Q_1) \cap T(Q_2)=\emptyset$.

The first important property for the collection $\calT$ is that it admits a ``tree structure".

\begin{prop} [\cite{HW20}] \label{Btree}
Let $\calD$ be a dyadic grid on $\mathbf b \Omega$ and $\calT$ be the dyadic extension associated to it. Then there exists a collection of points in $\Omega$, which is denoted as $\calC$, satisfying the following properties:
\begin{enumerate}
    \item [(1).] The set $\calC$ has a one-to-one correspondence with $\calD$ (and therefore $\calT$). Moreover, we can write
    $$
    \calC=\bigcup_{k \ge 1} \bigcup_{j=1}^{J_k} \{c_j^k\},
    $$
    where there exists a unique $Q \in \calD$, such that
    \begin{enumerate}
        \item [$\bullet$] $\pi(c_j^k)=c(Q_j^k)$;
        \item [$\bullet$] $\dist\left(c(Q_j^k), c_j^k \right)=\del^k/2$;
    \end{enumerate}
    \item [(2).] $\calC$ has a tree structure, that is, we say $c_i^{k'}$ is an offspring of $c_j^k$ for some $i, j>0$ and $k'>k>0$, if $\pi(c_i^{k'}) \in Q_j^k$. 
\end{enumerate}
\end{prop}

Using Proposition \ref{Btree}, we have the following definition. 

\begin{defn}
\begin{enumerate}
    \item [(1).] We refer each $c_j^k$ constructed in Proposition \ref{Btree} as the \emph{center} of the tent $T(Q_j^k)$ for some $Q_j^k \in \calD_k$ and $\calC$ as a \emph{Bergman tree} associated to the dyadic grid $\calD$;
    \item [(2).] For each $Q_j^k \in \calD$, we let 
    $$
    T^{\calK}(Q_j^k):=T(Q_j^k) \backslash \left(\bigcup_{c_i^{k'} \ \textrm{is an offspring of} \ c_j^k} T(Q_i^{k'}) \right)
    $$
    be the ``\emph{kube}" associated to it (see, Figure \ref{Figure1} below). Moreover, we let
    $$
    \calT^{\calK}(\mathbf b \Omega):=\Omega \backslash \left(\bigcup_{Q \in \calD} T(Q) \right).
    $$
\end{enumerate}
\end{defn}

\medskip

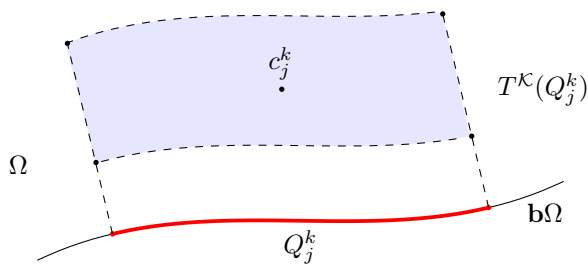
\begin{figure}[ht]
\begin{tikzpicture}[scale=5]
\draw   plot[smooth,domain=-.7:.7] (\x, {\x*(\x-0.2)*(\x+0.2)/3});
%\draw  [dashed] plot[smooth,domain=-.7:-.409175] (\x, {\x*\x/3+0.3});
\draw (.65, .07) node [below] {${\bf b}\Omega$};
%\draw  [dashed] plot[smooth,domain=.409175:.7] (\x, {\x*\x/3+0.3});
\draw (0, 0) node [below] {$Q_j^k$}; 
\draw [red, line width = 0.50mm] plot[smooth,domain=-.5:.5] (\x, {\x*(\x-0.2)*(\x+0.2)/3});
\draw (-.75, .1) node[above] {$\Omega$};
\fill (-0.544, 0.1551) circle [radius=.2pt]; 
\fill (0.455, 0.2251) circle [radius=.2pt]; 
\fill (.378, .5504) circle [radius=.2pt]; 
\fill (-.62, 0.4720) circle [radius=.2pt]; 
\draw  [dashed] plot[smooth,domain=-.53:.425] (\x, {\x*(\x-0.2)*(\x+0.2)/3+0.2});
\draw  [dashed] plot[smooth,domain=-.625:.357] (\x, {\x*(\x-0.2)*(\x+0.2)/3+0.545});
\fill [red] (.5, .035) circle [radius=.2pt];
\fill [red] (-.5, -0.035) circle [radius=.2pt];
\draw [dashed](-.5, -0.035)--(-.62, 0.4720);
\draw [dashed] (.5, 0.035)--(.378, .5504);
\fill [opacity=. 1, blue] plot[smooth,domain=-.53:.45] (\x, {\x*(\x-0.2)*(\x+0.2)/3+0.2})--(.378, .5504)--(-0.544, 0.1551);
\fill [opacity=. 1, blue] plot[smooth,domain=-.625:.357] (\x, {\x*(\x-0.2)*(\x+0.2)/3+0.545})--(.378, .5504)--(-0.544, 0.1551);
\fill  (-0.05, .35) circle [radius=.2pt];
\draw (-.053, .35) node [above] {$c_j^k$}; 
\draw (.5, .35) node [right] {$T^{\calK}(Q_j^k)$};  
\end{tikzpicture}
\caption{The kube $T^{\calK}(Q_j^k)$ associated to a dyadic cube $Q_j^k$ on $\mathbf b\Omega$. }
\label{Figure1}
\end{figure}

\medskip

We collect several useful facts about the kubes.

\begin{lem} [\cite{HW20}] 
Let $\calD$ and $\calT$ be defined as above. Then
\begin{enumerate}
    \item [(1).] $\pi\left(T(Q_j^k) \right)=\pi \left(T^{\calK}(Q_j^k) \right)=Q_j^k$, for each $j, k>0$;
    \item [(2).] The kubes are pairwise disjoint and $\bigcup\limits_{j, k} T^{\calK}(Q_j^k) \cup T^{\calK}(\mathbf b \Omega)=\Omega$;
    \item [(3).]  $\emph{\textrm{Vol}} \left(T^{\calK}(Q_j^k) \right) \simeq \emph{\textrm{Vol}} \left(T(Q_j^k) \right)$, where the implicit constant is independent of $j$ and $k$.
\end{enumerate}
\end{lem}

The second important property for the collection $\calT$ is that it forms a Muckenhoupt basis for the domain $\Omega$.  Let us recall some basic setup from \cite{CMP11}. Recall that by a \emph{basis} $\calB$ of $\Omega$, we mean by a collection of open sets contained in $\Omega$ and the maximal operator associated to $\calB$ is given by 
$$
\calM_{\calB} f(z):= \sup_{z \in \B \in \calB} \langle f \rangle_B
$$
if $z \in \Omega$ and $\calM_{\calB} f(z)=0$ otherwise, where for any $B \in \calB$, we denote
$$
\langle f \rangle_B:=\frac{1}{\textrm{Vol}(B)} \int_B |f(z)|dV(z). 
$$

\begin{defn}
A basis $\calB$ is a \emph{Muckenhoupt basis} if for each $p$, $1<p<\infty$, and for every $\omega=\omega(z) \ge 0$ on $\Omega$ with
$$
[\omega]_{p, \calB}:=\langle \omega \rangle_B \left \langle \omega^{1-p'} \right\rangle^{p-1}_B<\infty,
$$
one has
\begin{equation} \label{maximal}
\int_\Omega \left[\calM_{\calB}f(z)\right]^p \omega(z)dV(z) \lesssim \int_\Omega |f(z)|^p\omega(z)dV(z),
\end{equation} 
where the implicit constant is independent of $f$ and depends only on $[\omega]_{p, \calB}$, the defining function $\rho$ and the dimension $n$.
\end{defn}

\begin{thm} \label{Mbasis} 
$\calT$ is a Muckenhoupt basis associated to $\Omega$.
\end{thm}

\begin{proof}
The proof is similar to \cite[Proposition 3.13]{GHK20} and hence we omit the detail here. 
\end{proof}

\medskip

Finally, we study the relation between Bergman kernel and the tents and we have the following important estimate. The Boundary behavior of the Bergman kernel on a bounded, smooth, strictly  pseudoconvex domain is well understood. In fact, an asymptotic expansion for the Bergman kernel was established by Fefferman  \cite{Fefferman}. The lemma below only require an upper bound estimate for the kernel and the precise expansion is not necessary (see. e.g., \cite{McN03}).
\begin{lem}\label{thm4.2}
	Let $\Omega$ be a smooth, bounded, strictly pseudoconvex domain in $\mathbb C^n$ with a defining function $\rho$. For each $p_o\in \mathbf b\Omega$, there exists a neighborhood $U$ of $p_o$, holomorphic coordinates $(\zeta_1,\dots,\zeta_n)$ and a constant $C>0$, such that for $p,q\in U\cap \Omega$,
	\begin{align}
	|K_{\Omega}(p;\bar q)|\leq C\left(|\rho(p)|+|\rho(q)|+|p_n-q_n|+\sum_{j=1}^{n-1}|p_k-q_k|^2\right)^{-n-1}.
	\end{align}
	Here $p=(p_1,\dots,p_n)$ is in $\zeta$-coordinates.
\end{lem} 
Here the holomorphic coordinates $(\zeta_1,\dots,\zeta_n)$ can be chosen as the coordinates for the polydisc $D(p,r)$. Thus, there exists a constant $r_o$ such that the polydisc 
$D(\pi(p),r)$ contains points $p, q$ and satisfies 
\[\textrm{Vol}(D(\pi(p),r))\approx\left(|\rho(p)|+|\rho(q)|+|p_n-q_n|+\sum_{j=1}^{n-1}|p_k-q_k|^2\right)^{n+1}.\]
\begin{lem} \label{lem001}
For any $z, w \in \Omega$, there exists a dyadic cube $Q \in D^i$ for some $i \in \{1, \dots, K_0\}$ (see, Proposition \ref{adjac}), such that $z, \xi \in T(Q)$ and
$$
\left|K(z, \xi) \right| \lesssim \frac{1}{\textrm{Vol}(T(Q))},
$$
where the implicit constant in the above estimate only depends on the domain $\Omega$ and the dimension $n$. Here we recall that $K(\cdot, \cdot)$ is the Bergman kernel associated to $\Omega$. 
\end{lem}

\begin{proof}
The desired estimate is a consequence of the fact that the tent $T(B(\pi(p),r_o))$ and the set $D(\pi(p), r_o)\cap\Omega$ are comparable, i.e. there is a constant $C>0$ so that \[C^{-1}T(B(\pi(p),r_o))\subseteq D(\pi(p), r_o)\cap\Omega\subseteq CT(B(\pi(p),r_o)).\] 
This containment follows from an triangle inequality argument (see, e.g., \cite[Lemma 2.7]{J10}).
\end{proof}

%--------------------------------------
\subsection{Sub-mean value property for kubes}

The second half of this section is devoted to show that the kubes constructed above possess a sub-mean value property. We prove this by showing that show each kube is comparable to a Kobayashi ball near the boundary. This generalizes such an observation for the unit ball in $\C^n$, which is due to Arcozzi, Rochberg and Sawyer \cite{ARS06} (see, also \cite{RTW17}). 

Let us briefly recall the definition and the main properties of the Kobayashi ball and we refer the reader \cite{Aba89, JP93, K98} for a more sophisticated treatment for this subject. 

Let $k_{\D}$ be the Poincar\'e distance on the unit disc $\D \subset \C$. Let further, the \emph{Lempert function} $\del_\Omega: \Omega \times \Omega \to \R^+$ of $\Omega$ to be defined by:
\begin{eqnarray*}
\del_\Omega(z, \xi):%
&=& \inf \big\{k_{\D}(\zeta, \eta) \big | \ \textrm{there exists a holomorphic function} \ \phi: \D \to X \ \\
&& \quad \quad \quad \quad \quad \quad \quad \textrm{with} \ \phi(\zeta)=z \ \textrm{and} \ \phi(\eta)=\xi \big \},
\end{eqnarray*}
for all $z, \xi \in \Omega$. The \emph{Kobayashi distance} $k_\Omega: \Omega \times \Omega \to \R^+$ is the largest pseudodistance on $\Omega$ bounded above by $\del_\Omega$. Note that since $\Omega$ is bounded, $k_\Omega$ is a true distance (see, e.g., \cite[Theorem 2.3.14]{Aba89}). For any $z \in \Omega$ and $r \in (0, 1)$, the \emph{Kobayashi ball} centered at $z$ with radius $\frac{1}{2} \log \frac{1+r}{1-r}$ is given by
\begin{equation} \label{Kball}
B_\Omega(z, r):= \left\{z' \in \Omega: \tanh k_\Omega(z, z')<r \right\}. \end{equation} 

We are ready to state the main result in the second half of this section.

\begin{thm} \label{KobayashiKube}
There exists $0<\alpha, \beta<1$, such that for any kube $T^{\calK}(Q_j^k), j, k \ge 1$, one has
$$
B_\Omega (c_j^k, \alpha) \subseteq T^{\calK}(Q^k_j) \subseteq B_\Omega (c_j^k, \beta),
$$
where we recall that $c_j^k$ is the center of $T(Q_j^k)$. 
\end{thm}

\begin{proof}
Let $z\in \Omega$ be a point near the boundary. Let $(z_1,\dots,z_n)$ be a holomorphic coordinate system centered at $z$ so that $\re z_1$-axis is  the outward normal direction and $z_2,\dots, z_n$ are complex tangent directions at point $z$. 
By the well-known asymptotic behavior of the Kobayashi metric on strictly  pseudoconvex domain  (see, e.g., \cite{Aladro, Graham, Sibony}), $B_\Omega(z, r)$  is comparable to a polydisc $D_r(z)$ of the form:
\[D_r(z)=\{w\in\Omega:|w_1-z_1|<c_1t,|w_2-z_2|<c_2\sqrt{t},\dots,|w_n-z_n|<c_2\sqrt{t}\}.\]
Here $c_1$ and $c_2$ are constants that only depend on $r$ and $t=\dist(z, \mathbf b\Omega)$. Hence, it suffices to show that there exists $0<\alpha,\beta<1$ such that
$$
D_\alpha (c_j^k) \subseteq T^{\calK}(Q^k_j) \subseteq D_\beta (c_j^k).
$$
We first show that 
$$
D_\alpha (c_j^k) \subseteq T^{\calK}(Q^k_j) .
$$ 
Since $Q^k_j$ is comparable to the quasi-metric ball $B(\pi(c^k_j), \delta^k) \subset \mathbf b\Omega$, we therefore prove the first inclusion with $Q_j^k$ replaced by $B(\pi(c^k_j), \delta^k)$.

Suppose $\zeta\in D_\alpha(c_j^k)$. Let $\zeta^\prime$ be the projection of $\zeta$ on the hypersurface  $\{w:\rho(w)=\rho(c^k_j)\}$. 
Then we have $d(\zeta,c^k_j)\lesssim \delta^{k}/2$ and 
\[d(\zeta^\prime,c^k_j)\lesssim d(\zeta^\prime,\zeta)+d(c^k_j,\zeta)\lesssim \delta^{k}.\]
Since $\delta$ is chosen to be sufficiently small and the gradient $\nabla \rho$ is continuous near $\mathbf b\Omega$ with $\left| \nabla \rho(\xi) \right|=1$ when $b \in \mathbf b\Omega$, we may assume that $|\nabla\rho(x)-\nabla \rho(y)|<\frac{1}{10}$ for all  $y\in \mathbf b\Omega$ and $x\in B(y,\delta)$.
Then, from $\pi(\zeta)=\zeta^\prime -\rho(c^k_j)\nabla\rho(\zeta)$ and $\pi(c^k_j)=c^k_j -\rho(c^k_j)\nabla\rho(c^k_j)$ , we obtain
\begin{align*}
d(\pi(\zeta),\pi(c^k_j))&\lesssim d\left(\pi(\zeta),\zeta^\prime-\rho (c^k_j)\nabla \rho(c^k_j)\right)+d\left(\pi(c^k_j),\zeta^\prime-\rho (c^k_j)\nabla \rho(c^k_j)\right)\nonumber\\&\lesssim -\epsilon \rho(c^k_j)+d\left(\pi(c^k_j),\zeta^\prime-\rho (c^k_j)\nabla \rho(c^k_j)\right).
\end{align*} 
Note that there exists a constant $C$ such that  $D_\alpha (c_j^k) \subseteq CD(\pi(c^k_j),\delta^k)$. Thus, $D_\alpha (c_j^k) -\rho (c^k_j)\nabla \rho(c^k_j)\subseteq CD(\pi(c^k_j),\delta^k)$ (see, Figure \ref{Figure2}). 

\bigskip

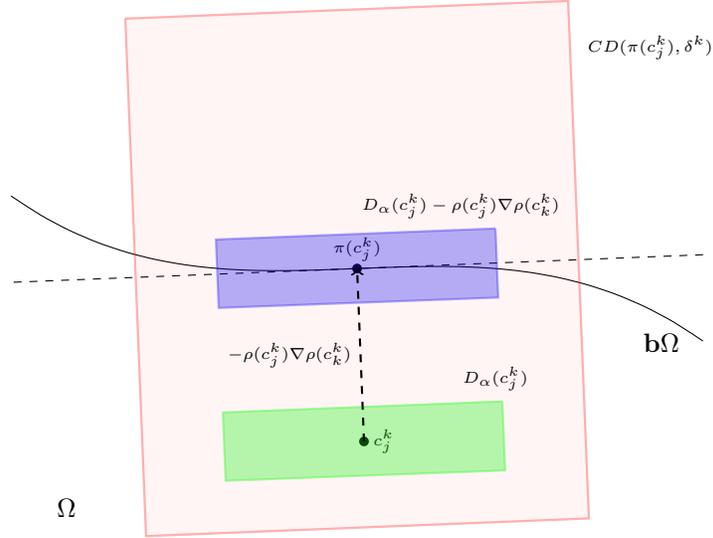
\begin{figure}[ht]
\begin{tikzpicture}[scale=9.2]
\draw   plot[smooth,domain=-.5:.5] (\x, {-\x*(\x-0.2)*(\x+0.2)+0.4});
\fill (0, 0.4) circle [radius=.2pt];
\draw (0, 0.398) node [above] {\tiny{$\pi(c_j^k)$}}; 
\fill (0.01, 0.15) circle [radius=.2pt];
\draw (0.01, 0.15) node [right] {\tiny{$c_j^k$}}; 
\draw [dashed] (0.5, 0.42)--(-0.5, 0.38); 
\draw [opacity=. 3, blue, line width = 0.30mm](0.2, 0.458)--(0.204, 0.358)--(-0.2, 0.343)--(-0.204, 0.442)--(0.2, 0.458); 
\draw [opacity=. 3, green, line width = 0.30mm] (0.21, 0.208)--(0.214, 0.108)--(-0.19, 0.093)--(-0.194, 0.192)--(0.21, 0.208); 
\draw (0.2, 0.21) node [above] {\tiny{$D_\alpha(c_j^k)$}}; 
\draw [opacity=.3, red, line width=0.30mm]   (-0.335, 0.7614) -- (0.305, 0.787) -- (0.335, 0.0386) -- (-0.305, 0.013) -- (-0.335, 0.7614);
\fill [opacity=. 3, blue](0.2, 0.458)--(0.204, 0.358)--(-0.2, 0.343)--(-0.204, 0.442)--(0.2, 0.458);
\draw (0.15, 0.46) node [above] {\tiny{$D_\alpha(c_j^k)-\rho(c_j^k) \nabla \rho(c_k^k)$}}; 
\fill [opacity=. 3, green](0.21, 0.208)--(0.214, 0.108)--(-0.19, 0.093)--(-0.194, 0.192)--(0.21, 0.208);  
\fill [opacity=.04, red]   (-0.335, 0.7614) -- (0.305, 0.787) -- (0.335, 0.0386) -- (-0.305, 0.013) -- (-0.335, 0.7614);  
\draw [->, dashed, thick] (0.01, 0.15) -- (0, 0.4); 
\draw (0.32, 0.72) node [right] {\tiny{$CD(\pi(c_j^k), \del^k)$}};
\draw (0.005, 0.275) node [left] {\tiny{$-\rho(c_j^k) \nabla \rho(c_k^k)$}}; 
\draw (0.44, 0.32) node [below] {${\bf b}\Omega$}; 
\draw (-.42, 0.08) node [below] {$\Omega$}; 
\end{tikzpicture}
\caption{$D_\alpha(c^k_j)-\rho(c^k_j)\triangledown\rho(c^k_j)\subseteq CD(\pi(c^k_j),\delta^k)$}
\label{Figure2}
\end{figure}

\medskip

Since the sets $D_\alpha (c_j^k) -\rho (c^k_j)\nabla \rho(c^k_j)$ and $CD(\pi(c^k_j),\delta^k)$ have the common center $\pi(c^k_j)$, the containment still holds after applying a dilation to both sets. Therefore, there exists a constant $C_2$ such that
\[C^{-1}C_2^{-1}D_\alpha (c_j^k) -\rho (c^k_j)\nabla \rho(c^k_j)\subseteq C^{-1}D(\pi(c^k_j), \delta^{k})\subseteq D(\pi(c^k_j),
\delta^{k}).\]
We may choose $\alpha_1$ to be smaller so that 
\[D_{\alpha_1} (c_j^k) \subseteq C^{-1}C_2^{-1}D_\alpha (c_j^k) .\]
Then we obtain $D_{\alpha_1} (c_j^k)-\rho (c^k_j)\nabla \rho(c^k_j)\subseteq D(\pi(c^k_j),\delta^{k})$ which also implies that 
\[d(\pi(c^k_j),\zeta^\prime-\rho (c^k_j)\nabla \rho(c^k_j))\approx \delta^{k}.\]
Hence $d(\pi(\zeta),\pi(c^k_j))\approx \delta^{k}$ and $\pi(D_{\alpha_1}(c^k_j))\subseteq B(\pi(c^k_j),\delta^{k})$ for some $\alpha_1$.

Now we are left to show that the polydisc $D_\alpha(c_j^k)$ can be chosen to also have the correct distance to the boundary $b\Omega$, i.e.
\begin{align*}
\{-\rho(z):z\in D_\alpha(c_j^k)\}\subseteq \{-\rho(z):z\in T^{\calK}(Q^k_j)\}.
\end{align*}
 By  perhaps choosing smaller constant $\alpha$, we may assume that the distance between sets $D_\alpha(c_j^k)$ and the boundary of $T^{\calK}(Q^k_j)$ is comparable to $\delta^k$. Then there exists constant $c$ such that the distance between \[\{z\pm c\triangledown \rho(c^k_j):z\in D_\alpha(c_j^k), \rho(z)=\rho(c^k_j)\}\subseteq T^{\calK}(Q^k_j).\] By perhaps choosing a smaller $\alpha$, we have $$D_\alpha(c_j^k)\subseteq\{z\pm c\triangledown \rho(c^k_j):z\in D_\alpha(c_j^k), \rho(z)=\rho(c^k_j)\},$$ which implies
 $$
 D_\alpha (c_j^k) \subseteq T^{\calK}(Q^k_j) .
 $$ 

Suppose $\zeta\in T^{\calK}(Q^k_j)$. Then $\text{dist}(\zeta,b\Omega)\in [\delta^{k+1},\delta^k]$ and $\pi(\zeta)\in B(\pi(c^k_j),c\delta^k)$ for some $c$. Recall that 
$\zeta^\prime$ is the projection of $\zeta$ on the hypersurface  $\{w:\rho(w)=\rho(c^k_j)\}$. 
Then we have $d(\zeta,c^k_j)\lesssim \delta^{k}/2$ and 
\[d(c^k_j,\zeta)\lesssim d(\zeta^\prime,\zeta)+d(\zeta^\prime,c^k_j)\lesssim \delta^k+d(\zeta^\prime,c^k_j).\]
By $\pi(\zeta)\in B(\pi(c^k_j),c\delta^k)$, we have \[\zeta^\prime=\pi(\zeta)+\rho(c_j^k)\triangledown\rho(\zeta)=\pi(\zeta)+\rho(c^k_j)\triangledown(\rho(\zeta)-\rho(c^k_j))+\rho(c^k_j)\triangledown\rho(c^k_j).\]
Since $|\nabla \rho(\zeta)-\nabla \rho(c^k_j)|<\frac{1}{10}$, 
$$
\zeta^\prime\in B(\pi(c^k_j),c_1\delta^k)+\rho(c^k_j)\triangledown\rho(c^k_j)
$$
for some slightly larger constant $c_1$. Note that $\dist(c^k_j,\pi(c^k_j)=\delta^k/2$, it follows that $D(c^k_j,\delta^k)$ and $D(\pi(c^k_j),\delta^k)$ are of comparable size. Hence
$\zeta^\prime\in B(c^k_j,c_2\delta^k)$ for some constant $c_2$. As a result
$d(c^k_j,\zeta)\lesssim \delta^k$ and
$\zeta\in D_\beta(c^k_j)$ for some constant $\beta$.

\end{proof}

\begin{cor} \label{sub-mean} 
The following estimate holds: for each dyadic cube $Q_j^k$ and any non-negative plurisubharmonic function $f$, 
$$
f(z) \lesssim \frac{1}{\emph{\textrm{Vol}}\left(T^{\calK}(Q_j^k)\right)} \int_{B_\Omega(c_j^k, \widetilde{\beta})} f(w) dV(w), \quad z \in T^{\calK}(Q_j^k),
$$
where $\beta>0$ is defined in Theorem \ref{KobayashiKube} and $\widetilde{\beta}=\frac{1+\beta}{2}$.  
\end{cor}

\begin{proof}
The above estimate is a consequence of Theorem \ref{KobayashiKube} and the sub-mean value property of plurisubharmonic functions on Kobayashi balls on a bounded strictly  pseudoconvex domain (see, e.g., \cite[Corollary 2.7 and Corollary 2.8]{AS11}). Therefore, we would like to leave the detail to the interested reader. 
\end{proof}

\bigskip
%-----------------------------------------
\section{Dyadic Carleson embedding on strictly  pseudoconvex domains}

The goal of this section is to study the dyadic Carleson embedding on the domain $\Omega$ via the dyadic decomposition constructed in the previous section. This part generalizes such a characterization in the upper half plane and unit ball (see, \cite{HLSW19}), moreover, it also can be viewed as a dyadic  counterpart of \cite[Theorem 3.4]{AS11}.

Let $\Omega$ be a bounded strictly   pseudoconvex domain and  $\calD^t, t=1, \dots, K_0$ be the adjacent grids associated to it (see, Proposition \ref{adjac}). 

\begin{defn}
Let $\lambda>0$. We say a finite Borel measure $\mu$ defined on $\Omega$ is a \emph{dyadic $\lambda$-Carleson measure} if
$$
\sup_{\substack{Q \in \calD^i \\ i \in \{1, \dots, K_0\}}} \frac{\mu(T(Q))}{\textrm{Vol}(T(Q))^\lambda}<\infty,
$$
and a \emph{vanishing dyadic $\lambda$-Carleson measure} if
$$
\lim_{c\left(T(Q)\right) \to \mathbf b\Omega} \frac{\mu(T(Q))}{\textrm{Vol}(T(Q))^\lambda}=0. 
$$
We denote $\calC\calM_d(\lambda)$ ($\calV\calC\calM_d(\lambda)$, respectively) to be the collection of all dyadic $\lambda$-Carleson measures (vanishing dyadic $\lambda$-Carleson measures) on $\Omega$. 

\end{defn}

\begin{lem}\label{lem3.2}
For $s,t>0$ and $\mu$ be any positive Borel measure on $\Omega$, the following statements are equivalent:
	\begin{enumerate}
	 \item [(1).] For any $f\in A^t$, 
	\[\left(\int_\Omega |f|^{s} d\mu\right)^{\frac{1}{s}} \lesssim \left(\int_\Omega |f|^t dV \right)^{\frac{1}{t}}.\]
	\item [(2).] For any $l>0$, and any $f\in A^{tl}$, 
	\[\left(\int_\Omega |f|^{sl} d\mu\right)^{\frac{1}{sl}} \lesssim \left(\int_\Omega |f|^{tl} dV \right)^{\frac{1}{tl}}.\]
	\end{enumerate} 
\end{lem}
\begin{proof}
The proof follows from a slight modification of \cite[Theorem 4]{CM95} and we omit the details here. 
\end{proof}

Given $q \ge 1$, $u \in H(\Omega)$ and $\varphi: \Omega \to \Omega$ a holomorphic mapping, we define the measure $\mu_{u, \varphi, q}$ by
$\mu_{u, \varphi, q}(E):=(|u|^qV)\left(\varphi^{-1}(E)\right)$ for any $E \subseteq \Omega$, that is, for any measurable function $f$, we have
$$
\int_\Omega f d\mu_{u, \varphi, q}=\int_\Omega f \circ \varphi(z) |u(z)|^qdV(z). 
$$

We have the following Carleson type result. 

\begin{thm} \label{bounded}
Let $q \ge p \ge 1$, $u$ and $\varphi$ be defined as above. Then the following statements are equivalent:
\begin{enumerate}
    \item [(1).] $\mu_{u, \varphi, q} \in \calC\calM_d \left(\frac{q}{p} \right)$; 
    \item [(2).] $W_{u, \varphi}: A^p \to A^q$ is bounded;
    \item [(3).] The \emph{$\frac{q}{p}$-Berezin transform} of $\mu_{u, \varphi, p}$ is bounded, namely,
    $$
    \sup_{z \in \Omega} B_{\frac{q}{p}}\mu_{u, \varphi, q}(z):=\sup_{z \in \Omega} \int_\Omega \frac{|K(\xi, z)|^{\frac{2q}{p}}}{K(z, z)^{\frac{q}{p}}} d\mu_{u, \varphi, q}(\xi)<\infty. 
    $$
\end{enumerate}
\end{thm}

\begin{proof}
$(1) \Longrightarrow (2)$. We let $\calD$ be any dyadic grid on $\mathbf b\Omega$ and recall that 
$$
\bigcup_{Q \in \calD} T^{\calK}(Q) \cup T^{\calK}(\mathbf b\Omega)=\Omega. 
$$
This suggests for any $f \in A^p$, we can bound
\begin{eqnarray*}
\|W_{u, \varphi} f\|_q^q%
&=& \int_\Omega |u(z)|^q |f(\varphi(z))|^qdV(z)= \int_\Omega |f(z)|^q d\mu_{u, \varphi, q}(z) \\
&\le& \sum_{Q \in \calD} \int_{T^{\calK}(Q)}|f(z)|^q d\mu_{u, \varphi, q}(z).
\end{eqnarray*}
Here we only deal with the case when $Q \neq \mathbf b\Omega$, while the case when $Q=\mathbf b\Omega$ can be estimated in a similar way and hence we would like to leave the detail to the interested reader. 

By applying Corollary \ref{sub-mean} twice, (i) and the fact that $q \ge p$, we have
\begin{eqnarray*}
&&\sum_{Q \in \calD} \int_{T^{\calK}(Q)}|f(z)|^q d\mu_{u, \varphi, q}(z) \\
&& \quad \quad \quad \lesssim \sum_{Q \in \calD} \int_{T^{\calK}(Q)} \left[\frac{1}{\textrm{Vol}\left(T^{\calK}(Q_j^k)\right)} \int_{B_\Omega(c_j^k, \widetilde{\beta})} |f(w)|^q dV(w) \right] d\mu_{u, \varphi, q}(z) \\
&& \quad \quad \quad  = \sum_{Q \in \calD} \frac{\mu_{u, \varphi, q}(T^{\calK}(Q)}{\textrm{Vol}(T^{\calK}(Q_j^k))} \int_{B_\Omega(c_j^k, \widetilde{\beta})} |f(w)|^qdV(w) \\
&& \quad \quad \quad  \lesssim \sum_{Q \in \calD} \int_{B_\Omega(c_j^k, 2\beta)} \left[\textrm{Vol}\left(T^{\calK}(Q_j^k)\right)\right]^{\frac{q-p}{p}} |f(w)|^{q-p} |f(w)|^pdV(w) \\
&&\quad \quad \quad  \lesssim \|f\|_p^{q-p} \sum_{Q \in \calD} \int_{B_\Omega (c_j^k, \widetilde{\beta})} |f(w)|^pdV(w) \\
&& \quad \quad \quad  \lesssim \|f\|_p^p,
\end{eqnarray*}
where in the last estimate, we have used the fact that the set $\{B_\Omega(c_j^k, \widetilde{\beta})\}_{j, k \ge 1}$ has finite overlap, which follows from the facts that the set $T^{\calK}(Q_j^k)$ are pairwise disjoint for different $(j, k)$ and for each $c_j^k$, it has only finitely many children $c_{j'}^{k+1}$ in the tree structure, since so does $c(Q_j^k)$ (see Theorem \ref{HKdecomp}). 

\medskip

$(2) \Longrightarrow (3)$. Since $W_{u, \varphi}: A^p \to A^{q}$,  it holds that
$$
\left(\int_\Omega |f|^q d\mu_{u,\varphi,q}\right)^{\frac{1}{q}} \lesssim \left(\int_\Omega |f|^p dV \right)^{\frac{1}{p}} \quad \forall f\in A^p(\Omega).
$$
By Lemma \ref{lem3.2} with $t=p$ and $s=q$, the above inequality is equivalent to 
$$
\left(\int_\Omega |f|^{\frac{2q}{p}} d\mu_{u,\varphi,q}\right)^{\frac{p}{2q}} \lesssim \left(\int_\Omega |f|^{2} dV \right)^{\frac{1}{2}} \quad \forall f\in A^{2}(\Omega).
$$
Now for any $z \in \Omega$, we consider the normalized kernel 
$
k_z(\xi):={K(\xi, z)}/{K(z, z)^{{1}/{2}}}.
$
It is clear that $k_z \in A^2$ and $\|k_z\|_2=1$ for any $z \in \Omega$. Then substituting $k_z$ into the inequality above yields
$$
  \sup_{z \in \Omega} \left(B_{\frac{q}{p}}\mu_{u, \varphi, q}(z)\right)^{\frac{p}{2q}}=\sup_{z \in \Omega}  \left(\int_\Omega |k_z|^{\frac{2q}{p}} d\mu_{u,\varphi,q}\right)^{\frac{p}{2q}} \lesssim \sup_{z \in \Omega} \left(\int_\Omega |k_z|^{2} dV \right)^{\frac{1}{2}}=1.
$$
Thus $\sup\limits_{z \in \Omega} B_{\frac{q}{p}}\mu_{u, \varphi, q}(z)\lesssim 1$.
\medskip

$(3) \Longrightarrow (1)$. Let $Q \in \calD^i$ and $c_Q$ be the center of the tent $T(Q)$. Then by (3), we have 
\begin{eqnarray*}
1%
&\gtrsim& \int_\Omega \frac{|K(\xi, c_Q)|^{\frac{2q}{p}}}{K(c_Q, c_Q)^{\frac{q}{p}}} d\mu_{u, \varphi, q}(\xi) \ge \int_{T(Q)} \frac{|K(\xi, c_Q)|^{\frac{2q}{p}}}{K(c_Q, c_Q)^{\frac{q}{p}}} d\mu_{u, \varphi, q}(\xi) \\
&\simeq& |K(c_Q, c_Q)|^{\frac{q}{p}} \cdot \mu_{u, \varphi, q}(T(Q)) \simeq \frac{\mu_{u, \varphi, q}(T(Q))}{\textrm{Vol}(T(Q))^{\frac{q}{p}}}, 
\end{eqnarray*}
which clearly implies (i). Here, in the above estimates, we have used the fact that for any $Q \in \calD^i$, 
$$
|K(\xi, c_Q)| \simeq |K(c_Q, c_Q)| \simeq \frac{1}{\textrm{Vol}(T(Q))},
$$
which is a consequence of Theorem \ref{KobayashiKube} and \cite[Lemma 2.1, Lemma 3.1 and Lemma 3.2]{AS11}. 
\end{proof}

\begin{rem}
When $\Omega$ is the unit disc in $\C$ or the unit ball in $\C^n$, one can prove the (3) by simply testing $K(\cdot, z)^{\frac{2}{p}}$, since it is holomorphic and belongs to $A^p$. However, this is in general not true when $\Omega$ is a strictly  pseudoconvex domain, as the off-diagonal Bergman kernel $K(\cdot, \cdot)$ might be equal to $0$. 
\end{rem}

For the compactness of $W_{u, \varphi}$, we have the following result. 

\begin{thm} \label{compact}
Let $q \ge p \ge 1$, $u$ and $\varphi$ be defined as above. Then the following statement are equivalent:
\begin{enumerate}
    \item [(1).] $\mu_{u, \varphi, q} \in \calV\calC\calM_d\left(\frac{q}{p} \right)$; 
    \item [(2).] $W_{u, \varphi}: A^p \to A^q$ is compact;
    \item [(3).] The $\frac{q}{p}$-Berezin of $\mu_{u, \varphi, q}$ vanishes on $\mathbf b\Omega$, namely
    $$
    \lim_{z \to \mathbf b\Omega} B_{\frac{q}{p}}\mu_{u, \varphi, q}(z)=\lim_{z \in \mathbf b\Omega} \int_\Omega \frac{|K(\xi, z)|^{\frac{2q}{p}}}{K(z, z)^{\frac{q}{p}}} d\mu_{u, \varphi, q}(\xi)=0. 
    $$
\end{enumerate}
\end{thm}

\begin{proof}
The proof of this result is standard and follows from an easy modification of Theorem \ref{bounded}. Therefore, we omit it here. 
\end{proof}

%---------------------------------------------------

\section{Sparse domination of weighted composition operators on strictly  pseudoconvex domains}

In this section, we study the behavior of the weighted composition operator by using the sparse domination technique from harmonic analysis. As a consequence, we establish a new weighted type estimate of $W_{u, \varphi}$ acting on Bergman spaces, which seems not to be covered the classical Carleson measure technique. 

We begin with the following integral representation: $f \in A^1$, one has
\begin{equation} \label{integralreps}
f(z)=\int_\Omega K(z, w)f(w)dV(w), \quad z \in \Omega. 
\end{equation}
It is clear that \eqref{integralreps} holds when $p=2$; while for general $p \ge 1$, \eqref{integralreps} follows from the fact that $A^2$ is dense in $A^1$. 

\subsection{Boundedness.} In the first part of this section, we study the boundedness of $W_{u, \varphi}$ by using sparse domination. 

The following pointwise bound is important.

\begin{lem} \label{pointwise}
Let $q \ge p \ge 1$, $N \in \N$ with $1 \le N \le p$ and $\mu \in \calC\calM_d \left(\frac{q}{p} \right)$. Then for $z \in \Omega$, 
\begin{eqnarray} \label{pointwisesp}
\int_\Omega |f(w)|^{q-N} |K(z, w)|d\mu(w)%
&\lesssim& \sum_{i=1}^{K_0} \sum_{Q \in \calD^i} \frac{\one_{T(Q)}(z)}{\emph{\textrm{Vol}(T(Q))}} \cdot \emph{\textrm{Vol}}(T(Q))^{\frac{q}{p}-1} \nonumber \\
&& \quad \quad \quad\quad \times \int_{T(Q)} |f(z)|^{(q-N)}dV(z). 
\end{eqnarray}
\end{lem}

\begin{proof}
Note that by Lemma \ref{lem001}, we have for any $z \in \Omega$, 
\begin{equation} \label{eq01}
\textrm{LHS of} \ \eqref{pointwisesp} \lesssim \sum_{i=1}^{K_0} \sum_{Q \in \calD^i} \frac{\one_{T(Q)}(z)}{\textrm{Vol}(T(Q))} \int_{T(Q)} |f(w)|^{q-N} d\mu(w).  
\end{equation}
We further bound the term 
\begin{equation} \label{eq02}
\int_{T(Q)} |f(w)|^{(q-N)} d\mu(w)
\end{equation}
by the tree structure on $\Omega$. More precisely, let $c_Q$ be the center of $T(Q)$ and by Corollary \ref{sub-mean},  we have
\begin{eqnarray*} 
\eqref{eq02}%
&=& \sum_{c_{Q'}: \ \textrm{an offspring of} \ c_Q} \int_{T^{\calK}(Q')} |f(w)|^{q-N} d\mu(w) \nonumber \\
&\le& \sum_{c_{Q'}: \ \textrm{an offspring of} \ c_Q} \frac{\mu\left(T^{\calK}(Q')\right)}{\textrm{Vol}(T^{\calK}(Q'))} \int_{B_\Omega(c(Q'), \beta)} |f(w)|^{q-N} dV(w) \nonumber \\
\end{eqnarray*}
Since $\mu \in \calC\calM_d \left(\frac{q}{p} \right)$, we further have 
\begin{eqnarray} \label{eq03}
	&& \sum_{c_{Q'}: \ \textrm{an offspring of} \ c_Q} \frac{\mu\left(T^{\calK}(Q')\right)}{\textrm{Vol}(T^{\calK}(Q'))} \int_{B_\Omega(c(Q'), \beta)} |f(w)|^{q-N} dV(w) \nonumber \\
&\lesssim& \sum_{c_{Q'}: \ \textrm{an offspring of} \ c_Q} \left[\textrm{Vol}(T^{\calK}(Q'))\right]^{\frac{q}{p}-1} \int_{B_\Omega(c(Q'), \beta)} |f(w)|^{q-N} dV(w)\nonumber  \\
&\le& \sup_{c_{Q'}: \ \textrm{an offspring of} \ c_Q} \left[\textrm{Vol}(T^{\calK}(Q'))\right]^{\frac{q}{p}-1} \nonumber \\
&& \quad \quad \quad \quad \quad \quad  \times\sum_{c_{Q'}: \ \textrm{an offspring of} \ c_Q} \int_{B_\Omega(c(Q'), \beta)} |f(w)|^{q-N} dV(w) \nonumber \\
&\lesssim& \left[\textrm{Vol}(T^{\calK}(Q))\right]^{\frac{q}{p}-1} \cdot \int_{T^{\calE}(Q)} |f(z)|^{q-N} dV(w).
\end{eqnarray}
Here 
\begin{equation} \label{2001}
T^{\calE}(Q):=\bigcup_{c_{Q'}: \ \textrm{an offspring of} \ c_Q} B_\Omega(c_{Q'}, \beta),
\end{equation} 
and the last inequality uses the fact that $\{B_\Omega(c_Q, \beta)\}_{Q \in \calD^i}$ has finitely many overlaps. 

Combining \eqref{eq01} and \eqref{eq03}, we find that
\begin{equation} \label{eq05}
\textrm{LHS of} \ \eqref{pointwisesp} \lesssim   \sum_{i=1}^{K_0} \sum_{Q \in \calD^i} \frac{\one_{T(Q)}(z)}{\textrm{Vol}(T(Q))}  \cdot  \left[\textrm{Vol}(T^{\calK}(Q))\right]^{\frac{q}{p}-1} \int_{T^{\calE}(Q)} |f(z)|^{q-N} dV(w) 
\end{equation}
To this end, we wish to replace the integral domain $T^{\calE}(Q)$ in the above estimate by $T(\widetilde{Q})$ for some $\widetilde{Q} \in \calD^j$. To see this, we let $\ell(Q)=\del^k$ for some $k \ge 0$. Following a similar argument as in Theorem \ref{KobayashiKube}, we see that there exists some constant $C>0$, which is independent of $k$, such that 
$$
T^{\calE}(Q) \cap \mathbf b\Omega \subset B(c(Q), C\del^k), 
$$
where we recall $c(Q) \in \mathbf b\Omega$ is the center of $Q$ and $B(c(Q), C\del^k)$ is defined as in \eqref{bball}. Therefore, by Proposition \ref{adjac}, there exists dyadic cube $\widetilde{Q} \in \calD^j$ for some $j \in \{1, \dots, K_0\}$, such that
$$
T^{\calE}(Q) \subset \widetilde{Q} \quad \textrm{and} \quad \ell(\widetilde{Q}) \lesssim \del^k, 
$$
which further implies that
\begin{eqnarray*}
&&\frac{\one_{T(Q)}(z)}{\textrm{Vol}(T(Q))} \cdot \left[\textrm{Vol}(T^{\calK}(Q))\right]^{ \frac{q}{p}-1} \cdot \int_{T^{\calE}(Q)} |f(z)|^{q-N} dV(w)  \\
&& \quad \quad \quad \quad  \lesssim \frac{\one_{T(\widetilde{Q})}(z)}{\textrm{Vol}(T(\widetilde{Q}))} \cdot \left[\textrm{Vol}(T^{\calK}(\widetilde{Q}))\right]^{\frac{q}{p}-1} \cdot \int_{T^{\calE}(\widetilde{Q})} |f(z)|^{q-N} dV(w).
\end{eqnarray*}
Combining the above estimate with \eqref{eq05} then yields the desired estimate \eqref{pointwisesp}. 
\end{proof}

We are ready to state our main results for the boundedness. 

\begin{thm} \label{thm001}
Let $q \ge p \ge 1$, $u \in H(\Omega)$ and $\varphi:\Omega \to \Omega$ be a holomorphic mapping. Let further, $W_{u, \varphi}: A^p \to A^q$ be bounded. Then for any $\gamma \ge 1$and $f \in A^p$,
\begin{equation} \label{eq11}
\|W_{u, \varphi} f\|_q^q \lesssim \inf_{N \in \N, 1\le N \le p} \left( \sum_{i=1}^{K_0} \sum_{Q \in \calD^i} \left[\emph{\textrm{Vol}}(T(Q)) \right]^{\frac{q}{p}} \langle |f|^N \rangle_{T(Q)} \cdot \langle |f|^{q-N} \rangle_{T(Q)} \right). 
\end{equation}
\end{thm}

\begin{proof}
Let $f \in A^p$ and we let $N \in \N$ with $1 \le N \le p$. Then using \eqref{integralreps} and the fact that $A^\frac{p}{N} \subset A^1$, we have
\begin{eqnarray*}
\|W_{u, \varphi} f\|_q^q%
&=& \int_\Omega |u(z)|^q \left|C_\varphi f(z) \right|^qdV(z)\\
&=& \int_\Omega |u(z)|^q \left| f \left(\varphi(z)\right) \right|^{q-N} \left|C_\varphi(f^N)(z) \right|dV(z) \\
&\lesssim& \int_\Omega |u(z)|^q |f \circ \varphi(z)|^{q-N}  \left(\int_\Omega |f(w)|^N |K(\varphi(z), w)| dV(w) \right) dV(z) \\
&=& \int_\Omega |f(w)|^N \left(\int_\Omega |u(z)|^q |f\circ \varphi(z)|^{q-N} |K(\varphi(z), w)| dV(z) \right) dV(w) \\
&=& \int_\Omega |f(w)|^N \left(\int_\Omega |f(z)|^{q-N} |K(z, w)|d\mu_{u, \varphi, q}(z) \right) dV(w). 
\end{eqnarray*}
We can now bound the interior integral using Lemma \ref{pointwise}, and clearly this implies the desired result. 
\end{proof}

Here, we call the estimate \eqref{eq11} a \emph{sparse domination estimate} associated to the operator $W_{u, \varphi}$. Below we provide some of its application. 

\begin{thm} \label{thm002}
Let $q \ge 1$, $u \in H(\Omega)$ and $\varphi:\Omega \to \Omega$ be a holomorphic mapping. Then the following statements are equivalent:
\begin{enumerate}
\item [(\textit{i}).] $W_{u, \varphi}: A^q \to A^q$ is bounded;
\item [(\textit{ii}).] For any $f \in A^q$, 
\begin{equation} \label{sparse-01}
\|W_{u, \varphi} f\|_q^q \lesssim \inf_{N \in \N, 1\le N \le p} \left(\sum_{i=1}^{K_0} \sum_{Q \in \calD^i} \emph{\textrm{Vol}}(T(Q)) \langle |f|^N \rangle_{T(Q)} \langle |f|^{q-N} \rangle_{T(Q)} \right). 
\end{equation} 
\end{enumerate}
\end{thm}

\begin{proof}
The assertion $(i)$ implies $(ii)$ is clear by taking $p=q$ and $\gamma=1$ in Theorem \ref{thm001}. Therefore, it suffices for us to show $(ii)$ implies $(i)$. Let us fix any $N \in \N$ with $1 \le N < q$ (when $q=N$, the desired result follows from a similar argument and we leave the detail to the interested reader). Then for any $f \in A^q$, we can find an $i_0 \in \{1, \dots, K_0\}$, such that
\begin{eqnarray*} 
&& \sum_{i=1}^{K_0} \sum_{Q \in \calD^i} \textrm{Vol}(T(Q)) \langle |f|^N \rangle_{T(Q)} \langle |f|^{q-N} \rangle_{T(Q)}  \\
&& \quad \quad \quad  \quad \quad \quad \quad \le K_0 \sum_{Q \in \calD^{i_0}} \textrm{Vol}(T(Q)) \langle |f|^N \rangle_{T(Q)} \langle |f|^{q-N} \rangle_{T(Q)}.
\end{eqnarray*}
Therefore, we have
\begin{eqnarray*}
\|W_{u, \varphi}f \|_q^q%
&\lesssim& \sum_{Q \in \calD^{i_0}} \textrm{Vol}(T(Q)) \langle |f|^N \rangle_{T(Q)} \langle |f|^{q-N} \rangle_{T(Q)} \\
&\lesssim& \sum_{Q \in \calD^{i_0}} \textrm{Vol}(T^{\calK}(Q)) \langle |f|^N \rangle_{T(Q)} \langle |f|^{q-N} \rangle_{T(Q)} \\
&\lesssim& \int_\Omega \calM_{\calT^{i_0}} \left(|f|^N\right) \calM_{\calT^{i_0}} \left(|f|^{q-N} \right) dV(z) \\
\end{eqnarray*}
Since $\{T^{\calK}(Q)\}$ are pairwise disjoint, we have
\begin{eqnarray*}
\|W_{u, \varphi}f \|_q^q%
&\le& \left( \int_\Omega \calM_{\calT^{i_0}}(|f|^N)^{\frac{q}{N}}dV(z)  \right)^{\frac{N}{q}} \cdot \left(\int_\Omega \calM_{\calT^{i_0}}(|f|^{q-N})^{\frac{q}{q-N}} dV(z) \right)^{\frac{q-N}{q}} \\
&\lesssim& \left(\int_\Omega |f|^{N \cdot \frac{q}{N}}dV(z) \right)^{\frac{N}{q}} \cdot \left( \int_\Omega |f|^{(q-N) \cdot \frac{q}{q-N}}dV(z) \right)^{\frac{q-N}{q}} \\
&=& \|f\|_q^q, 
\end{eqnarray*} 
where in the above estimates, we recall that $\calT^{i_0}$ is the dyadic extension of $\calD^{i_0}$ inside $\Omega$, $\calM_\calT$ is the maximal operator associated to the Muckenhoupt basis $\calT$ and we have used the estimate \eqref{maximal} with $\omega \equiv 1$ in our last estimate. 
\end{proof}

\begin{rem}
When $\Omega$ is the upper half plane, one can avoid using the maximal operator $\calM_\calT$ associated to a Muckenhoupt basis by simply replacing it with the standard uncentered Hardy-Littlewood maximal operator, which clearly enjoys the strong $(p, p)$ estimate. The reason for us to use $\calM_{\calT}$ here is that Kobayashi balls do not necessarily satisfy the ``doubling property" (see, e.g., \cite[Lemma 1.23]{Zhu04} for the unit ball case) and hence the Hardy-Littlewood maximal operator over all Kobayashi balls might fail the strong $(p, p)$-estimate. 
\end{rem}

We now turn to the case when $p<q$. To begin with, we defined the \emph{fractional maximal operator} associated to $\calT$ as follows: for any $0<\alpha<1$, 
$$
\calM_{\calT, \alpha}f(z)= \sup_{z \in T(Q) \in \calT} \textrm{Vol}(T(Q))^\alpha \langle f \rangle_{T(Q)}. 
$$

\begin{lem} \label{fractional}
For any $0<\alpha<1$ and $q>p>1$ satisfying $\frac{1}{p}-\frac{1}{q}=\alpha$, $\calM_{\calT, \alpha}: L^p(\Omega) \to L^q(\Omega)$ is bounded.  
\end{lem}

\begin{proof}
This result is standard, and to be self-contained, we include the proof here. Note that it suffices to prove $\calM_{\calT, \alpha}: L^p(\Omega) \to L^{q, \infty}(\Omega)$ is bounded and the strong $(p, q)$ estimate then follows from interpolation. 

Let $t>0$ and $f \in L^p(\Omega)$. Note that if $z \in \Omega$ such that $\calM_{\calT, \alpha}f(z)>t$, then there exists a maximal tent $T(Q)$ such that 
$$
\textrm{Vol}(T(Q))^\alpha \langle f \rangle_{T(Q)}>t.
$$
Denote $\calT_t$ to be set of maximal disjoint tent $T(Q)$ with this property. Then we have
\begin{eqnarray*}
&& t^q \textrm{Vol} \left( \left\{z \in \Omega: \calM_{\calT, \alpha}f(z)>t \right\} \right) = \sum_{T(Q) \in \calT_t} \textrm{Vol}(T(Q)) \cdot t^q \\
&& \quad \quad \quad \quad \le \sum_{T(Q) \in \calT_t} \textrm{Vol}(T(Q)) \cdot \left(  \textrm{Vol}(T(Q))^\alpha \langle f \rangle_{T(Q)} \right)^q \\
&& \quad \quad \quad \quad =\sum_{T(Q) \in \calT_t} \left[\textrm{Vol}(T(Q)) \right]^{1+\alpha q-q} \left(\int_{T(Q)} |f(z)|dV(z) \right)^q 
\end{eqnarray*}
By H\"older's inequality, we can bound the last term above by
\begin{eqnarray*}
&&  \quad \quad \quad \quad \le \sum_{T(Q) \in \calT_t} \left[\textrm{Vol}(T(Q)) \right]^{1+\alpha q-q} \left(\int_{T(Q)} |f(z)|^pdV(z) \right)^\frac{q}{p} \cdot  \left[\textrm{Vol}(T(Q)) \right]^{\frac{q}{p'}} \\ 
&& \quad \quad \quad \quad = \sum_{T(Q) \in \calT_t} \left(\int_{T(Q)} |f(z)|^pdV(z) \right)^\frac{q}{p} \le \left( \sum_{T(Q) \in \calT_t} \int_{T(Q)} |f(z)|^pdV(z)  \right)^{\frac{q}{p}} \\
&& \quad \quad \quad \quad \le \left(\int_\Omega |f(z)|^pdV(z) \right)^{\frac{q}{p}},
\end{eqnarray*}
where in the above estimate, we have used the identity
$$
1+\alpha q-q+\frac{q}{p'}=0. 
$$
The desired claim then follows by taking the supremum over $t>0$ in the above estimate. 
\end{proof}

\begin{thm} \label{thm004}
Let $q>p \ge 1$ and $u$ and $\varphi$ be defined as in Theorem \ref{thm001}. Denote
$$
\Z_{p, q}:=\{N \in \N: N \ge 1, N<p<q<p+N\} \neq \emptyset.
$$
Then the following statements are equivalent:
\begin{enumerate}
    \item [\textit{(i)}.] $W_{u, \varphi}: A^p \to A^q$ is bounded;  
    \item [\textit{(ii)}.] For any $f \in A^p$, 
    \begin{equation} \label{sparse-02}
    \|W_{u, \varphi}f\|_q^q \lesssim \inf_{N \in \Z_{p, q}} \left(\sum_{i=1}^{K_0} \sum_{Q \in \calD^i} \left[\emph{\textrm{Vol}}(T(Q)) \right]^{\frac{q}{p}} \langle |f|^N \rangle_{T(Q)}  \cdot \langle |f|^{q-N} \rangle_{T(Q)} \right). 
    \end{equation} 
\end{enumerate}
Note that $\Z_{p, q}$ is not empty in general, for example, when both $p$ and $q$ are large while $q-p$ is small. 
\end{thm}

\begin{proof}
Again, it suffices to show $(ii)$ implies $(i)$. Since $\Z_{p, q}$ is not empty, we have $p<q<2p$ and hence $0<\frac{q}{p}-1<1$. Fix any $N \in \Z_{p, q}$ and denote
$$
s=\frac{p}{p+N-q} \quad \textrm{and} \quad s'=\frac{p}{q-N}.
$$
Note that $1<\frac{p}{N}<\frac{p}{q-p}$ and 
\begin{equation} \label{eq21}
\frac{N}{q}-\frac{1}{s}=\frac{q}{p}-1. 
\end{equation}
Let $i_0 \in \{1, \dots, K_0\}$ be the index such that \begin{eqnarray*} 
&& \sum_{i=1}^{K_0} \sum_{Q \in \calD^i} \left[\textrm{Vol}(T(Q))\right]^{\frac{q}{p}} \langle |f|^N \rangle_{T(Q)} \langle |f|^{q-N} \rangle_{T(Q)}  \\
&& \quad \quad \quad  \quad \quad \quad \quad \le K_0 \sum_{Q \in \calD^{i_0}} \left[\textrm{Vol}(T(Q))\right]^{\frac{q}{p}} \langle |f|^N \rangle_{T(Q)} \langle |f|^{q-N} \rangle_{T(Q)}.
\end{eqnarray*} 
Therefore, by Lemma \ref{fractional}, 
\begin{eqnarray*}
\|W_{u, \varphi}f\|_q^q%
&\lesssim& \sum_{Q \in \calD^{i_0}} \left[\textrm{Vol}(T(Q))\right]^{\frac{q}{p}} \langle |f|^N \rangle_{T(Q)} \langle |f|^{q-N} \rangle_{T(Q)} \\
&\lesssim& \sum_{Q \in \calD^{i_0}} \textrm{Vol}(T^{\calK}(Q)) \frac{\left[\textrm{Vol}(T(Q))\right]^{\frac{q}{p}-1}}{\textrm{Vol}(T(Q))} \int_{T(Q)} |f|^N dV(z) \cdot  \langle |f|^{q-N} \rangle_{T(Q)} \\
&\le& \int_\Omega \calM_{\calT, \frac{q}{p}-1}\left(|f|^N\right) \calM_{\calT} \left(|f|^{q-N} \right) dV(z) \\
&\le& \left(\int_\Omega \calM_{\calT, \frac{q}{p}-1}\left(|f|^N \right)^s dV(z) \right)^{\frac{1}{s}} \cdot \left( \int_\Omega \calM_{\calT} (|f|^{q-N})^{s'}dV(z) \right)^{\frac{1}{s'}}\\
&\le& \left(\int_\Omega |f|^{N \cdot \frac{p}{N}}dV(z) \right)^{\frac{N}{p}}\cdot \left(\int_\Omega |f|^{s'(q-N)}dV(z) \right)^{\frac{1}{s'}} = \|f\|_p^q, 
\end{eqnarray*}
where in the last estimate, we have used \eqref{eq21}.The proof is complete. 
\end{proof}

%-------------------------------------------------
\subsection{Compactness}
In the second part of this section, we study the compactness of $W_{u, \varphi}$ via the sparse domination estimate \eqref{eq11}. Note that when we characterize the boundedness of $W_{u, \varphi}$,  the estimate \eqref{eq11} does not contain any information from the pull-back measure $\mu_{u, \varphi, q}$ (see, Theorem \ref{thm002} and Theorem \ref{thm004}). This suggests the sparse domination estimates \eqref{sparse-01} and \eqref{sparse-02} are not enough to characterize the compactness of $W_{u, \varphi}$. The main idea is to distinguish the tents near $\mathbf b\Omega$ in the estimate \eqref{pointwisesp}. 

Here is the our main result for the compactness. 

\begin{thm} \label{thm003}
Let $q \ge 1$, $u$ and $\varphi$ be defined as in Theorem \ref{thm001}, and $W_{u, \varphi}$ be a bounded operator on  $A^q$. Then the following statements are equivalent.
\begin{enumerate}
    \item [(\textit{i}).] $W_{u, \varphi}: A^q \to A^q$ is compact;
    \item [(\textit{ii}).] Let $N \in \N$ with $1 \le N < q$. Then for any bounded set $\{f_m\}_{m \ge 1} \subset A^p$ with $\{f_m\}$ converges to $0$ uniformly on compact sets of $\Omega$, any $i \in \{1, \dots, K_0\}$ and any $Q \in \calD^i$, one has 
    \begin{eqnarray} \label{eq321}
   && \lim_{c_Q \to \mathbf b\Omega} \sup_{m \ge 1}  \frac{1}{\emph{\textrm{Vol}}(T(Q) \langle |f_m|^{q-N} \rangle_{T^{\calE}(Q)}}  \nonumber \\
   && \quad \quad \quad \quad \quad \quad \quad \cdot \sum_{c_{Q'}: \ \textrm{offspring of} \ c_Q} \mu_{u, \varphi, q} \left(T^{\calK}(Q') \right) \langle |f_m(z)|^{q-N} \rangle_{B_\Omega(c(Q'), \beta)}=0, 
    \end{eqnarray}
    where we recall that $c_Q$ is the center of the tent $T(Q)$, $T^{\calE}(Q)$ is defined as in \eqref{2001} and $\beta \in (0, 1)$ is defined in Theorem \ref{KobayashiKube}. 
\end{enumerate}
\end{thm}

\begin{proof}
$(i) \Longrightarrow (ii).$  Suppose $W_{u, \varphi}$ is compact, and therefore by Theorem \ref{compact}, $\mu_{u, \varphi, q} \in \calV\calC\calM_d(1)$. Without loss of generality, we may assume that $\sup_{m \ge 1} \|f_m\|_q=1$. 

For any $\varepsilon>0$, we can find some $K_0 \in \N$, such that for any $k>K_0$, one has for any $Q\in \calD^i_k$ with $\dist(c_Q, \mathbf b\Omega)<\del^k$ that
$$
\frac{\mu_{u, \varphi, q}(T(Q))}{\textrm{Vol}(T(Q))}<\varepsilon.
$$
Therefore, for any $k>K_0, m \ge 1$, any $i \in \{1, \dots, K_0\}$ and $Q \in \calD^i_k$, we have
\begin{eqnarray*}
&&\sum_{c_{Q'}: \ \textrm{offspring of} \ c_Q} \mu_{u, \varphi, q} \left(T^{\calK}(Q) \right) \langle |f_m(z)|^{q-N} \rangle_{{B_\Omega(c(Q'), \beta)}} \\
&& \quad \quad \quad \quad \le   \sum_{c_{Q'}: \ \textrm{offspring of} \ c_Q} \mu_{u, \varphi, q} \left(T(Q) \right) \langle |f_m(z)|^{q-N} \rangle_{{B_\Omega(c(Q'), \beta)}}  \\
&&\quad \quad \quad \quad   < \varepsilon \cdot \sum_{c_{Q'}: \ \textrm{offspring of} \ c_Q} \textrm{Vol} \left(T(Q) \right) \langle |f_m(z)|^{q-N} \rangle_{{B_\Omega(c(Q'), \beta)}}  \\
&&\quad \quad \quad \quad   \simeq \varepsilon \cdot \sum_{c_{Q'}: \ \textrm{offspring of} \ c_Q}  \int_{{B_\Omega(c(Q'), \beta)}} |f_m(z)|^{q-N}dV \\
&& \quad \quad \quad \quad  \lesssim \varepsilon \cdot \textrm{Vol}(T(Q)) \langle |f_m|^{q-N} \rangle_{T^{\calE}(Q)}, 
\end{eqnarray*}
where  in the last estimate above, we have use the fact that $\{B_\Omega(c(Q), \beta)\}_{Q \in \calD^i}$ has finite overlaps. The desired claim then follows by taking the supremum over all $m \ge 0$ first and then letting $\varepsilon$ converge to $0$. 

\medskip

$(ii) \Longrightarrow (i).$ Let $\{f_m\}_{m \ge 0} \subset A^q$ be a bounded set satisfying $f_m$ converges to uniformly $0$ on compact subsets of $\Omega$. It is well know that to prove $W_{u, \varphi}$ is compact, we only need to argue 
\begin{equation} \label{eq211}
\|W_{u, \varphi} f_m\|_q \to 0 \quad \textrm{as} \quad m \to \infty.
\end{equation}
To begin with, we might assume that $\varphi(\Omega)$ is not contained in any compact subset of $\Omega$, otherwise \eqref{eq211} is clearly.

Now for each $\ell \ge 1, m \ge 1$ and $1 \le N<q$, by the proof of Theorem \ref{thm001}, we have
\begin{equation} \label{1001}
\|W_{u, \varphi} f_m\|_q^q \le \int_\Omega |f_m(w)|^N \left(\int_\Omega
 |f_m(z)|^{q-N} |K(z, w)|d\mu_{u, \varphi, q}(z) \right)dV(w) 
 \end{equation}
Take any $\varepsilon>0$ and denote
\begin{equation} \label{1002}
A_m(w):=\int_\Omega
 |f_m(z)|^{q-N} |K(z, w)|d\mu_{u, \varphi, q}(z),
\end{equation}
which, by \eqref{eq01}, is clearly bounded by 
\begin{equation} \label{1003}
\sum_{Q \in \calD^i} \frac{\one_{T(Q)}(w)}{\textrm{Vol}(T(Q))} \int_{T(Q)} |f_m(z)|^{q-N} d\mu_{u, \varphi, q}(z),
\end{equation}
for some $i \in \{1, \dots, K_0\}$. 

For any $\ell \ge 0$, denote 
$$
E_\ell=\bigcup_{1 \le k \le \ell, Q \in \calD^i_k} T^{\calK}(Q) \quad \textrm{and} \quad 
\widetilde{E_\ell}:=\bigcup_{\substack{1 \le k \le \ell \\ Q \in \calD_k }} \overline{B_\Omega (c_Q, \beta)}. 
$$
where $\beta \in (0, 1)$ is the constant defined in Theorem \ref{KobayashiKube}. Note that it is clear that $E_\ell \subset \widetilde{E_\ell}$. Moreover, $\widetilde{E_\ell}$ is also a compact subset contained in $\Omega$. 

For any $\ell \ge 0$  and $Q \in \calD^i$, we have
\begin{eqnarray} \label{1004}
&&\int_{T(Q)} |f_m(z)|^{q-N} d\mu_{u, \varphi, q}(z) \nonumber \\
&& \quad \quad \quad  = \sum_{c_{Q'}: \ \textrm{offspring of} \ c_Q} \int_{T^{\calK}(Q')} |f_m(z)|^{q-N} d\mu_{u, \varphi, q}(z) \nonumber \\
&& \quad \quad \quad  \le \sum_{c_{Q'}: \ \textrm{offspring of} \ c_Q} \frac{\mu_{u, \varphi, q} \left(T^{\calK}(Q')\right)}{\textrm{Vol}(T^{\calK}(Q'))} \int_{B_\Omega(c(Q'), \beta)} |f_m(z)|^{q-N} dV(z) \nonumber \\
&& \quad \quad \quad :=B_{1, m, \ell}(Q)+B_{2, m, \ell}(Q),
\end{eqnarray} 
where 
$$
B_{1, m, \ell}(Q):=\sum_{\substack{c_{Q'}: \ \textrm{offspring of} \ c_Q \\ T(Q') \cap E_\ell=\emptyset }} \frac{\mu_{u, \varphi, q}\left(T^{\calK}(Q')\right)}{\textrm{Vol}(T^{\calK}(Q'))} \int_{B_\Omega(c(Q'), \beta)} |f_m(z)|^{q-N} dV(z)
$$
and
$$
B_{2, m, \ell}(Q):=\sum_{\substack{c_{Q'}: \ \textrm{offspring of} \ c_Q \\ T(Q') \cap E_\ell \neq \emptyset }} \frac{\mu_{u, \varphi, q} \left(T^{\calK}(Q')\right)}{\textrm{Vol}(T^{\calK}(Q'))} \int_{B_\Omega(c(Q'), \beta)} |f_m(z)|^{q-N} dV(z).
$$
Therefore, combining \eqref{1001}, \eqref{1002} and \eqref{1004}, we have
\begin{equation} \label{3000}
\|W_{u, \varphi} f_m\|_q^q \le \int_\Omega |f_m(w)|^N A_m(w) dV(w) \lesssim C_{1, m, \ell}+C_{2, m, \ell}, 
\end{equation}
where 
$$
C_{1, m, \ell}:=\sum_{Q \in \calD_i} \langle |f_m|^N \rangle_{T(Q)} B_{1, m, \ell}(Q)
$$
and
$$
C_{2, m, \ell}:=\sum_{Q \in \calD_i} \langle |f_m|^N \rangle_{T(Q)} B_{2, m, \ell}(Q).
$$
We control $C_{2, m, \ell}$ first. Recall that since $\mu_{u, \varphi, q} \in \calC\calM_d(1)$, for any $\ell \ge 1$, we have
\begin{eqnarray*}
B_{2, m, \ell}(Q)%
&\lesssim& \sum_{\substack{c_{Q'}: \ \textrm{offspring of} \ c_Q \\ T(Q') \cap E_\ell \neq \emptyset }} \int_{B_\Omega(c(Q'), \beta)} |f_m(z)|^{q-N} dV(z) \\
&\lesssim& \int_{T^{\calE}(Q) \cap \widetilde{E_\ell}} |f_m(z)|^{q-N}dV(z).
\end{eqnarray*}
Since $\{f_m\}$ converges to $0$ uniformly on compact subsets of $\Omega$, we can take $m$ sufficiently large, such that 
$$
\int_{T^{\calE}(Q) \cap \widetilde{E_\ell}} |f_m(z)|^{q-N}dV(z)<\varepsilon \cdot \textrm{Vol}\left(T^{\calE}
(Q)\right) \lesssim \varepsilon \cdot \textrm{Vol} \left(T(Q) \right)
$$
Therefore, following a similar argument as in Theorem \ref{thm002}, 
\begin{eqnarray} \label{3001}
C_{2, m, \ell} %
&\lesssim& \varepsilon \cdot \sum_{Q \in \calD_i} \textrm{Vol}  \left(T(Q) \right) \cdot \langle |f_m|^N \rangle_{T(Q)} \nonumber \\
&\lesssim& \varepsilon \cdot \sum_{Q \in \calD_i} \textrm{Vol}  \left(T^{\calK}(Q) \right) \cdot \langle |f_m|^N \rangle_{T(Q)} \nonumber 
\\
&\le& \varepsilon \int_\Omega \calM_{\calT^i} \left(|f_m|^N \right)dV(z) \lesssim \varepsilon, 
\end{eqnarray}
where we recall that $\calT^i$ is the dyadic extension associated to $\calD^i$. 

Next, we bound the term $C_{1, m, \ell}$. By our assumption \eqref{eq321}, there exists a $L_0 \in \N$, such that for any $\ell>L_0, m \ge 1$ and $Q \in \calD^i$ with $T(Q) \cap E_\ell=\emptyset$,  
$$
\sum_{c_{Q'}: \ \textrm{offspring of} \ c_Q} \mu_{u, \varphi, q} \left(T^{\calK}(Q') \right) \langle |f_m(z)|^{q-N} \rangle_{B_\Omega(c(Q'), \beta)}<\varepsilon \textrm{Vol}(T(Q)) \langle |f_m|^{q-N} \rangle_{T^{\calE}(Q)}. 
$$
Fix such a choice of $\ell$. Then for any $m \ge 1$, we have
$$
C_{1, m, \ell} \le \varepsilon \cdot \sum_{Q \in \calD^i} \textrm{Vol}(T(Q)) \langle |f_m|^N \rangle_{T(Q)} \langle |f_m|^{q-N} \rangle_{T^{\calE}(Q)}.
$$
Following a similar argument in the proof of Lemma \ref{pointwise}, we can further bound the last term in the above estimate by
$$
\varepsilon \cdot \sum_{t=1}^{K_0} \sum_{Q \in \calD^t}\textrm{Vol}(T(Q)) \langle |f_m|^N \rangle_{T(Q)} \langle |f_m|^{q-N} \rangle_{T^(Q)},
$$
which, again by the proof of Theorem \ref{thm002}, gives 
\begin{equation} \label{4000}
C_{1, m, \ell} \lesssim \varepsilon.
\end{equation} 
The desired claim \eqref{eq211} then clearly follows from \eqref{3000}, \eqref{3001} and \eqref{4000}. 
\end{proof}

\medskip
%----------------------------------------
\subsection{Weighted estimates for $W_{u, \varphi}$.}  In the last part of this section, we establish a new weighted estimate associated to the weighted composition operators acting on Bergman spaces on strictly  pseudoconvex domain.

\begin{defn}
Given $q \ge 1$, $u \in H(\Omega)$ and $\varphi: \Omega \to \Omega$, the weight class $\mathbf B_{u, \varphi}^q$ is defined to the collection of all non-negative functions $\omega$ on $\Omega$ satisfying
$$
[\omega]_{\mathbf B_{u, \varphi}^q}:=\sup_{\xi \in \Omega} \int_\Omega |u(z)|^q \omega(z) |K(\varphi(z), \xi)|dV(z)<\infty. 
$$
\end{defn}

We make several remarks for $\mathbf B_{u, \varphi}^q$. 
\begin{rem}
\begin{enumerate}
    \item [(1).] Here, $\mathbf B$ means that we define such a weight via a  \emph{``Bergman projection"-like transformation}, which is clear from the definition;
    \item [(2).] In general, if $\omega \in \mathbf B_{u, \varphi}^q$, the measure $\omega dV$ might not be Carleson (see, \cite[Remark 4.8]{HLSW19} for an example) and hence it is of its own interest;
    \item [(3).] The weight class $\mathbf B_{u, \varphi}^q$ can be interpreted as a version of Sawyer-testing condition. Such a condition was first introduced by Sawyer \cite{S82} in 1982 to study the behavior of Hardy-Littlewood maximal operators acting on weighted $L^p$ spaces, and later, the same idea has been applied by many authors in studying other function spaces and operators, such as \cite{APR17, PRW19}. 
    
\end{enumerate}
\end{rem}

We have the following result. 

\begin{thm} 
Let $q>1, 1<s<q'$, $u$ and $\varphi$ be defined as in Theorem \ref{thm001} and $\omega^{s'} \in \mathbf B_{u, \varphi}^q$, where $q'$ is the conjugate of $q$ (respectively, $s'$ is the conjugate of $s$). Let further, $W_{u, \varphi}$ be a bounded operator on $A^q$. Then the following weighted estimate hold
\begin{equation} \label{weightest01}
\int_\Omega \left|W_{u, \varphi}f (z) \right|^q \omega(z)dV(z) \lesssim \left[ \omega^{s'} \right]_{B_{u, \varphi}^q}^{\frac{1}{s'}} \|f\|_q^q, 
\end{equation}
where the implicit constant in the above estimate is independent of the choice of $f$ and weight $\omega$. 
\end{thm}

\begin{proof}
The proof of the estimate \eqref{weightest01} follows from the spirit of Theorem \ref{thm001} and Theorem \ref{thm002}. Note that
\begin{eqnarray*} 
&&\textrm{LHS of \eqref{weightest01}} \nonumber \\
&\lesssim& \int_\Omega |f(\xi)| \left( \int_\Omega |f \circ \varphi(z)|^{q-1}|u(z)|^q \omega(z) |K(\varphi(z), \xi)|dV(z) \right)dV(\xi) \nonumber \\
&=& \int_\Omega |f(\xi)| \bigg(\int_\Omega |f\circ \varphi(z)|^{q-1}|u(z)|^{\frac{q}{s}} |K(\varphi(z), \xi)|^{\frac{1}{s}} \nonumber \\
&& \quad \quad \quad \quad \quad  \quad \quad \quad \cdot \ |u(z)|^{\frac{q}{s'}} \omega(z) |K(\varphi(z), \xi)|^{\frac{1}{s'}} dV(z) \bigg) dV(\xi) \nonumber \\
&\le& \int_\Omega |f(\xi)| \left(\int_\Omega |f \circ \varphi(z)|^{s(q-1)} |u(z)|^q |K(\varphi(z), \xi)|dV(z) \right)^{\frac{1}{s}} \nonumber  \\
&& \quad \quad \quad \quad \quad \quad \quad \quad \cdot \left(\int_\Omega |u(z)|^q \omega^{s'}(z) |K(\varphi(z), \xi)| dV(z) \right)^{\frac{1}{s'}} dV(\xi) \nonumber \\
&\le& \left[ \omega^{s'} \right]_{\mathbf B_{u, \varphi}^q}^{\frac{1}{s'}} \int_\Omega |f(\xi)| \left(\int_\Omega |f(z)|^{s(q-1)} |K(z, \xi)|d\mu_{u, \varphi, q}(z) \right)^{\frac{1}{s}} dV(\xi).
\end{eqnarray*}
Denote the integral in the above line by $I$. Using Lemma \ref{pointwise}, there exists some $i \in \{1, \dots, K_0\}$, such that
\begin{eqnarray*}
I%
&\lesssim& \int_\Omega |f(\xi)| \left(\sum_{Q \in \calD^i} \frac{\one_{T(Q)}(\xi)}{\textrm{Vol}(T(Q))} \int_{T(Q)} |f(z)|^{s(q-1)} dV(z) \right)^{\frac{1}{s}} dV(\xi) \\
&\lesssim& \int_\Omega |f(\xi)| \sum_{Q \in \calD^i} \frac{\one_{T(Q)}(\xi)}{\textrm{Vol}(T(Q))^{\frac{1}{s}}} \left(\int_{T(Q)} |f(z)|^{s(q-1)} dV(z) \right)^{\frac{1}{s}} dV(\xi) \\
&=& \sum_{Q \in \calD^i} \textrm{Vol}(T(Q)) \langle |f| \rangle_{T(Q)} \left[\langle |f|^{s(q-N)} \rangle_{T(Q)} \right]^{\frac{1}{s}} \\
&\lesssim& \int_{\Omega} \calM_{\calT^i}(|f|) \left( \calM_{\calT^i}(|f|^{(q-1)s}) \right)^{\frac{1}{s}} dV(z) \\
&\le& \left(\int_\Omega |\calM_{\calT^i}(|f|)|^q dV(z) \right)^{\frac{1}{q}} \cdot \left(\int_\Omega \left(\calM_{\calT^i}(|f|^{(q-1)s}) \right)^{\frac{q'}{s}} dV(z) \right)^{\frac{1}{q'}} \\
&\lesssim& \|f\|_q \cdot \|f\|_q^{q-1}=\|f\|_q^q,
\end{eqnarray*}
where in the last estimate, we use the fact that $s<q'$. The proof is complete. 
\end{proof}

\end{document}